%% file: main.tex
\documentclass[a4paper,12pt]{amsproc}
\usepackage{mathrsfs}
\usepackage{appendix}
\usepackage{amssymb} 
\usepackage{fancyhdr} 
\usepackage{charter}
\usepackage{typearea} 

\usepackage{dsfont}

\usepackage{color} 

\usepackage{accents} 
\usepackage{hyperref} 
\usepackage{enumitem}
\usepackage{setspace}
\usepackage{tikz}
\usetikzlibrary{cd}
\usepackage{tikz-cd} 
\usepackage{bbm} 
\usepackage{amsmath,amstext,amsthm,amscd} 
\newtheorem{theorem}{Theorem}[section]
\numberwithin{equation}{section}

\newtheorem{lemma}[theorem]{Lemma}
\newtheorem{proposition}[theorem]{Proposition}
\newtheorem*{claim}{Claim}
\newtheorem{corollary}[theorem]{Corollary}
\newtheorem{remark}[theorem]{Remark}

\theoremstyle{definition}
\newtheorem{definition}[theorem]{Definition}

\newtheorem{example}[theorem]{Example}
\numberwithin{equation}{section}

\input{notations.tex}

\makeatletter
\def\l@subsection{\@tocline{2}{0pt}{2.5pc}{5pc}{}}
\def\l@subsubsection{\@tocline{2}{0pt}{5pc}{7.5pc}{}}
\makeatother
\usepackage[margin=2.5cm]{geometry}

\title[Bergman Kernel Asymptotics for Semipositive Line Bundles]
{Bergman Kernel Asymptotics for Semipositive Line Bundles near
Curvature-Degenerate Points}
\author{Yueh-Lin Chiang } 
\thanks{The author is supported by a doctoral grant from the French Ministry
of Higher Education and Research (MESRI). This paper will form part of the
author's PhD thesis.}
\begin{document}
\begin{spacing}{1.5}

\begin{abstract}
We study the asymptotic behavior of Bergman kernels for high powers $L^k$ of a semipositive line
bundle $L$ over a Hermitian manifold $X$ of arbitrary dimension. At points where the curvature degenerates, the classical asymptotic expansion may fail. We introduce a good finite-type condition for such points, which implies that $L$ is big when $X$ is compact. If $p$ is of good finite type and lies in the ample locus of $L$, we establish localization and rapid off-diagonal decay of the Bergman kernel near $p$. When the corresponding finite-type model is quasihomogeneous, we further obtain a full local asymptotic expansion in the $\mathscr C^\infty$-topology. These results are consequences of a general framework based on two spectral assumptions: a localized mild spectral gap
for the Kodaira Laplacian and a spectral gap for the rescaled local model.

As an application, we study pullbacks of positive line bundles under generically finite holomorphic maps. At locally finite ramification points of corank one, we obtain full local asymptotic expansions and off-diagonal decay. Under an additional regularity condition, we establish uniform on-diagonal asymptotics along smooth ramification hypersurfaces. Finally, for non-quasihomogeneous models, we show that the leading-order on-diagonal asymptotics are determined
by the model weight, and we give an example on $\mathbb P^2$ in which a logarithmic correction appears:
$B_k(p,p)\sim\frac{3}{4\pi^2}\frac{k}{\log k}$.
\end{abstract}

\end{spacing}
\begin{spacing}{1}
\maketitle \tableofcontents
\end{spacing}
\setstretch{1.2}

\input{statements.tex}

\bibliographystyle{alpha}
\bibliography{my_reference}
\end{document}

%% file: notations.tex
\newcommand{\Cs}{\mathbb C}

\newcommand{\R}{\mathbb R}
\newcommand{\N}{\mathbb N}

\newcommand{\ii}{\mathrm i}

\newcommand{\db}{\bar\partial}

\newcommand{\la}{\langle}
\newcommand{\ra}{\rangle}
\newcommand{\To}{\rightarrow}

\newcommand{\Lam}{\Lambda}

\newcommand{\Cinf}{\mathscr C^\infty}

\DeclareMathOperator{\Spec}{Spec}
\DeclareMathOperator{\Dom}{Dom}
\DeclareMathOperator{\supp}{supp}
\DeclareMathOperator{\rank}{rank}
\DeclareMathOperator{\Ker}{Ker}

\DeclareMathOperator{\Op}{Op}

\DeclareMathOperator{\Id}{Id}

\newcommand{\vphi}{{\varphi}}

\newcommand{\tphik}{{\widetilde{\varphi}_{k}}}
\newcommand{\tomegak}{{\widetilde{\omega}_{k}}}

\newcommand{\phimodk}{{\varphi^{\Lam}_k}}
\newcommand{\phimod}{{\varphi_{\mathrm{mod}}}}

\newcommand{\tBoxk}{\widetilde{\Box}_k}
\newcommand{\Boxmodk}{\Box_{k}^\Lambda}

\newcommand{\Dk}{\mathcal{D}_k}
\newcommand{\tDk}{\widetilde{\mathcal D}_k}

\newcommand{\Hk}{\mathcal{H}_k}
\newcommand{\tHk}{\widetilde{\mathcal{H}}_k}
\newcommand{\Hmodk}{\mathcal{H}_k}

\newcommand{\tBk}{\widetilde{B}_k}
\newcommand{\Bmodk}{B_{k}^{\Lambda}}

\newcommand{\tsBk}{\widetilde{\mathcal{B}}_k}
\newcommand{\sBmodk}{\mathcal{B}_{k}^{\Lambda}}

\newcommand{\Bmod}{{B}_{\mathrm{mod}}}
\newcommand{\sBmod}{\mathcal{B}_{\mathrm{mod}}}
\newcommand{\hBk}{\widehat{\mathcal{B}}_k}

\newcommand{\Rk}{\mathcal{R}_k(\zeta)}
\newcommand{\tsRk}{\widetilde{\mathcal R}_k(\zeta)}

\newcommand{\tRk}{\widetilde{R}_k(\zeta)}
\newcommand{\Rlock}{{R}^{\mathrm{loc}}_k(\zeta)}
\newcommand{\Rglobk}{{R}^{\mathrm{glob}}_k(\zeta)}

\newcommand{\tchik}{{\widetilde{\chi}_k}}
\newcommand{\hchik}{{\widehat{\chi}_k}}
\newcommand{\chik}{{{\chi}_k}}

\newcommand{\Sk}{{S^\Lambda_k}}

\newcommand{\F}{\mathscr{F}}
\newcommand{\Cr}{\mathscr{C}^r}

\newcommand{\Qk}{\mathcal{Q}_k}

\newcommand{\even}{\mathrm{even}}
\newcommand{\Ram}{\operatorname{Ram}}
\newcommand{\Ps}{\mathbb{P}}

%% file: statements.tex
\section{Introduction and main results}
Let $(X,\omega)$ be a Hermitian manifold and $L$ be a  holomorphic line bundle equipped with a smooth Hermitian metric $h$. The asymptotic behavior of the Bergman kernels $B_k(\cdot, \cdot)$ associated with high tensor powers $L^k:=L^{\otimes k}$ has been extensively studied. When $X$ is compact and the Chern curvature form $\mathrm{i}\Theta_h(L)>0$, the Bergman function admits a full asymptotic expansion
\[
B_k(p,p)\sim k^n b_0(p)+k^{n-1}b_1(p)+k^{n-2}b_2(p)+\cdots
\]
in the $\mathscr{C}^{\infty}$-topology. The asymptotic theory of Bergman kernels originates in the work of Tian \cite{Tian1990OnAS} and was developed further by Catlin \cite{catlin1999bergman} and Zelditch \cite{zelditch1998szego}; see also Dai--Liu--Ma \cite{dai2004asymptotic,dai2006asymptotic}, Ma--Marinescu \cite{ma2007holomorphic}, and Berman--Berndtsson--Sj\"ostrand \cite{berman2008direct}. Moreover, Ma and Marinescu established exponential off-diagonal decay estimates for the Bergman kernel \cite{ma2015exponential}.

When the Hermitian line bundle is only semipositive, in the sense that
\[
    \ii\Theta_h(L)\geq 0,
\]
the problem becomes subtler, especially near points where the
curvature degenerates. Cho, Kamimoto and Nose obtained a full pointwise asymptotic expansion of the Bergman function under the assumptions that the model of the local weight takes positive values away from the origin, is torus-invariant and quasihomogeneous, and that global peak sections exist \cite{cho2011asymptotics}.  Subsequently, Hsiao and
Marinescu established a full local asymptotic expansion of the Bergman
kernel in the $\mathscr{C}^{\infty}$-topology on the positive-curvature
locus, under adjoint settings $L^{\otimes k}\otimes K_X$ or assuming the point is in the ample locus \cite{hsiao2014asymptotics}. More recently, Marinescu and Savale obtained pointwise on-diagonal asymptotic expansions of Bergman kernels on a compact Riemann surface. They assumed that the curvature
is semipositive and vanishes to finite order at every point \cite{marinescu2024bochner}. Their proof relies on a growing
spectral gap obtained from sub-Riemannian and subelliptic estimates. The
argument used to establish this spectral gap does not extend directly to
higher dimensions. In the CR setting, Hsiao and Savale obtained full asymptotic expansions of the Szeg\H{o} kernel on weakly pseudoconvex three-dimensional CR manifolds of finite type \cite{hsiao2022bergman}, which is closely related to the Riemann surface case via the circle bundle of $L^*$. 

At the level of measures, weaker but more general results are known. For big $L$ with $\ii\Theta_h(L)\geq 0$, $k^{-n}B_k(x,x)\,dV_\omega$ converges weakly to $\frac{1}{n!}\big(\frac{\ii}{2\pi}\Theta_h(L)\big)^n$ \cite{berman2011fekete}, which extends to non-big $L$ on projective $X$ after renormalization and push-forward along the Kodaira--Iitaka map \cite{finski2026kodaira}, and the off-diagonal Bergman kernel is small away from the diagonal in a measure-theoretic sense \cite{finski2026localization}. Related work on semipositive and big line bundles concerns Fubini--Study approximation and quantization
\cite{marinescu2024geometric,finski2025geometric,hou2025quantization},
partial equilibrium \cite{darvas2024volume}, and zeros of random holomorphic sections
\cite{coman2017equidistribution,bayraktar2024zeros,liu2025semipositive}.

In this paper, we develop a framework for Bergman kernel asymptotics based on two spectral assumptions: a localized mild spectral gap for the Kodaira Laplacian (cf. Definition~\ref{spectral gap}), which was first introduced in \cite{hsiao2014asymptotics}, and a spectral gap for the local model on $\mathbb{C}^n$ (cf. Definition~\ref{def:model-gap}). Under these hypotheses at a given point $p\in X$, we establish smooth localization of the Bergman kernel $B_k(x,y)$ near $p$, together with rapid off-diagonal decay when one variable lies in a shrinking neighborhood of $p$ and the other lies outside a slightly larger neighborhood, which is a pointwise counterpart of the measure-theoretic results above (cf. Theorem~\ref{Thm_Loc_Off}). If, in addition, the local model is quasihomogeneous, we obtain a full local asymptotic expansion in the $\Cinf$-topology after scaling (cf. Theorem~\ref{Thm_Full expansion}). In particular, the framework applies to the positive-curvature setting of Hsiao--Marinescu \cite{hsiao2014asymptotics} whenever their localized mild spectral gap is available. 
 
To provide concrete geometric conditions for semipositive line bundles, we introduce the notion of a good finite-type point with respect to the Hermitian metric $h$ (cf. Definition~\ref{def_finite}). When $X$ is compact, the existence of such a point implies that $L$ is big (cf. Corollary~\ref{cor_good_finite_big}). In Section~\ref{sec_de}, we further show that if $p$ is of good finite type and belongs to the ample locus $\operatorname{Amp}(L)$, then both spectral gap assumptions are satisfied. Consequently, we obtain localization and rapid off-diagonal decay near such points, as well as a full local asymptotic expansion when the corresponding finite-type model is quasihomogeneous. In particular, we obtain the pointwise asymptotic expansion  \[
B_k(p,p)\sim k^{\mu}(b_0+k^{-d_1 }b_1+k^{-d_2}b_2+\cdots),
\]where $0<\mu\leq n$, $b_0>0$, and the exponents $d_j\to \infty $ are described in Theorem~\ref{Thm_de}. In complex dimension one, the finite-type condition considered by Marinescu--Savale \cite{marinescu2024bochner} implies our good finite-type and ampleness condition, which yields a local version of their result (cf. Corollary~\ref{corr}).

As a geometric application, we consider generically finite holomorphic maps between compact complex manifolds and study the Bergman kernels of pullbacks of positive line bundles. At a locally finite ramification point of corank one and local multiplicity $m$, the Bergman kernel admits a full expansion with leading order $k^{n-1+1/m}$ and an intrinsic leading coefficient. We also exhibit a semipositive metric on $\mathcal O_{\mathbb P^2}(1)$ for which the Bergman function grows as $k/\log k$ at a degenerate point.

It is natural to ask whether the local asymptotic expansions can be made uniform along the curvature-degenerate locus, which may be singular in general. To obtain a uniform estimate near $p$, one needs to control the geometric behavior of the degenerate locus, as shown in the example of the ramification locus of a branched covering (see Proposition~\ref{prop.uniform}). This is subtle in the general case. 

Our approach is based on the Helffer--Sj\"ostrand formula and off-diagonal estimates for the resolvents of the Kodaira Laplacian. In contrast to the classical localization technique (see~\cite{dai2006asymptotic,ma2007holomorphic}), which relies on a spectral gap for the global operator together with finite propagation speed, our argument does not require a global spectral gap. Such a gap need not hold for semipositive line bundles in dimension at least two \cite{donnelly2003spectral}. We replace $\omega$ by a $k$-dependent Hermitian metric with the same volume form, which does not change the Bergman projection. After anisotropic rescaling, the corresponding Kodaira Laplacian is locally identified with an operator on $\mathbb{C}^n$ built from the extended rescaled data. By the Helffer--Sj\"ostrand formula, comparing resolvents shows that, near $p$, a spectral cutoff of the modified operator agrees with that of the extended operator up to $O(k^{-\infty})$. By the model gap, the latter cutoff is exactly the Bergman projection of the extended data. Finally, the localized mild spectral gap for the original operator shows that $B_k$ agrees with this localized projection up to $O(k^{-\infty})$.

\subsection{Setting and main results}
Let $(X,\omega)$ be a Hermitian manifold of complex dimension $n$. Let $L$ be a holomorphic
line bundle over $X$, and $h$ be a smooth Hermitian metric on $L$ with semipositive Chern curvature form,
\[
\ii\Theta_h(L)\geq 0 ,
\quad \text{everywhere on }\, X.
\]
Let $L^k:=L^{\otimes k}$ and $h^k:=h^{\otimes k}$, and let
\[
\db_k:\Cinf(X,L^k)\to\Cinf(X, \Lambda^{0,1}T^*X\otimes L^k)
\]
be the $L^k$-valued $\db$-operator in degree $(0,0)$. Let $\Box_k=\Box_{h^k,\omega}:=\db_k^*\db_k$ be the Kodaira
Laplacian on sections of $L^k$ (or $L^k$-valued $\db$-Laplacian on $(0,0)$-forms); see~\eqref{set:Box}. Let 
\[
B_k:=B_{h^k,\omega}:L^2_{h^k,\omega}(X,L^k)\to \Ker\db_k\cap L^2_{h^k,\omega}(X,L^k)
\]
be the Bergman projection, that is, the orthogonal projection onto $L^2$-holomorphic sections. We write $B_k(x,y)$ (also denoted by  $B_{h^k,\omega}(x,y)$) for its Schwartz kernel (cf.\eqref{Bh}). 

Fix $p\in X$. Let $z=(z_1,\ldots,z_n)$ be local holomorphic coordinates centered at $p$ on an open neighborhood $D$ of $p$, and let $e_L$ be a local holomorphic frame of $L$ over $D$. Write
$|e_L|_h^2=:e^{-2\varphi}$. Then, $\varphi$ is smooth and plurisubharmonic, i.e. $\ii\partial\db\vphi\geq 0$, on $D$. Its formal Taylor expansion at $p$ is written as
\begin{equation}\label{eq}
\varphi(z)\sim\sum_{\alpha,\beta\in\N_0^n}a_{\alpha,\beta}z^\alpha\bar z^\beta .  
\end{equation}After changing the local holomorphic frame if necessary, we shall assume that $\varphi(0)=0$ and the finite jet of $\varphi$ relevant to our analysis contains no nonconstant pluriharmonic terms $\{c\operatorname{Re}z^\alpha\}_{\alpha\in\mathbb N_0^n,c\in\R}$.

\begin{definition}\label{def_finite}
We say that $p$ is of \textbf{good finite type} with respect to $h$
if there exist local holomorphic coordinates $z=(z_1,\ldots,z_n)$
centered at $p$, a local holomorphic frame of $L$, and
$\mathcal M=(m_1,\ldots,m_n)\in\mathbb N^n$ such that the corresponding
weighted truncation of \eqref{eq} can be written as
\begin{equation}\label{eqq}
 \varphi_{\mathcal M}(z,\bar z)
:=
\sum_{\substack{\alpha,\beta\in\mathbb N_0^n\\
\sum_{j=1}^n(\alpha_j+\beta_j)/(2m_j)\leq1}}
a_{\alpha,\beta}z^\alpha\bar z^\beta
=
\sum_{j=1}^n P_j(z_j,\bar z_j)+\underbrace{Q(z,\bar z)}_{\text{mixed terms}},   
\end{equation}
where each $P_j(z_j,\bar z_j)$ is a nonzero real-valued
homogeneous polynomial of degree $2m_j$ in $z_j$ and
$\bar z_j$. For example, when $m_j=3$, $P_j$ may take the form
$|z_j|^6+\operatorname{Re}(z_j^2\bar z_j^4)$. We further require $Q$ to be a plurisubharmonic
polynomial consisting only of mixed terms.
 Here, we call a monomial $z^\alpha\bar z^\beta$ mixed if it involves
at least two distinct coordinate variables.

We say that $p$ is of \textbf{quasihomogeneous good finite type}
if every monomial $a_{\alpha,\beta}z^\alpha\bar z^\beta$ occurring in $Q$ has weighted degree one, namely $\sum_j\frac{\alpha_j+\beta_j}{2m_j}=1$.
If $Q\equiv0$, we say that $p$ admits a \textbf{decoupled model}.
\end{definition}

In Section~\ref{sec_de}, we will show that if $p\in X$ is of good
finite type, then there exists a point $x_1$ arbitrarily close
to $p$ such that $\ii\Theta_h(L)>0$ at $x_1$. In particular, if $X$ is compact, Siu's criterion implies
that $L$ is big (cf. \cite[p.175]{Siu1985}). It is therefore natural to consider whether $p$
belongs to the ample locus $\operatorname{Amp}(L)$ of $L$ (cf.~
Definition~\ref{ample}).

\begin{theorem}\label{Thm_de}
Let $(X,\omega)$ be a compact Hermitian  manifold and let $(L,h)$
be a Hermitian holomorphic line bundle with
$\ii\Theta_h(L)\geq0$. Assume that $p$ is of good finite type with the expression \eqref{eqq} in the coordinates $z=(z_1,\cdots,z_n):D\to \Cs^n$ and a frame $e_L$ of $L$. Set $\lambda_j:=\frac{1}{2m_j}$ and define  \begin{equation}\label{sk}
  S_k(z):=(k^{-\lambda_1}z_1,\cdots,k^{-\lambda_n}z_n)  
\end{equation} to be the anisotropic scaling map. If $p\in\operatorname{Amp}(L)$, the Bergman kernel $B_k(x,y)$ satisfies the following:

\noindent\textbf{(a) Localization and off-diagonal decay.} For $\delta>0$, let \[D_{k}:=\{x\in D;|z_j(x)|<k^{-\lambda_j},j=1,\cdots,n\},\] and  $D_{k,\delta}:=\{x\in D;|z_j(x)|<k^{-\lambda_j+\delta},j=1,\cdots,n\}$. For any $r\in\N$, we have 
\begin{equation}\label{poff}
|B_k(x,y)|_{\mathscr{C}^r}=O(k^{-\infty}),\quad \text{uniformly for }\, x\in D_k,\,\, y\in X\setminus D_{k,\delta}.
\end{equation} 
The same bound holds with $x$ and $y$ interchanged. Here, we denote by  $|\cdot|_{\Cr}$ the pointwise $\Cr$-norm (cf.~\eqref{Crs}). Moreover, there exists a sequence $\widehat{B}_k\in\Cinf(\Cs^n\times\Cs^n)$, which depends only on the germs of $h$ and $\omega$ at $p$, such that  \[|B_k(S_k(z),S_k(w))-\widehat{B}_k(z,w)|_{\Cr}=O(k^{-\infty}),\] locally uniformly on $\Cs^n\times \Cs^n$.

\noindent\textbf{(b) Asymptotic expansion.}
Assume further that $p$ is of \textbf{quasihomogeneous good finite
type}. Set
\begin{equation}\label{oder}
\mathcal O_{\mathcal M}
:=
\left\{
\sum_{\ell=1}^m
\left(
\sum_{j=1}^n\frac{\nu_j^{(\ell)}}{2m_j}-1
\right)
\,;\,
m\in\mathbb N,\ 
\nu^{(\ell)}\in\mathbb N_0^n,\ 
\sum_{j=1}^n\frac{\nu_j^{(\ell)}}{2m_j}\geq1
\right\}.    
\end{equation}
Let $
0<d_1<d_2<\cdots$ be the increasing enumeration of the positive elements of
$\mathcal O_{\mathcal M}$. Then, there exist kernels
\[
\Phi_{j}\in\Cinf(\Cs^n\times\Cs^n),\qquad j\geq1,
\]
such that
\begin{equation}\label{eq:ex}
B_k(S_k(z),S_k(w))
\sim
k^{\sum_{j=1}^n\frac1{m_j}}
\left(
\Bmod(z,w)
+
\sum_{j=1}^{\infty}k^{-d_j}\Phi_{j}(z,w)
\right),    
\end{equation}
locally uniformly in the $\Cinf$-topology on
$\Cs^n\times\Cs^n$. We refer to Theorem~\ref{Thm_Full expansion} for the precise meaning of \eqref{eq:ex} and the corresponding remainder estimates. Here, $\Bmod(z,w)$ denotes the Bergman kernel of the weighted Fock
space on $\Cs^n$ associated with the model weight
$\varphi_{\mathcal M}$ and the volume of $\omega$ at the point $p$. Moreover,
\[
\Bmod(0,0)>0.
\]If, in addition, the model is even, i.e.
$\varphi_{\mathcal M}(-z)=\varphi_{\mathcal M}(z)$, define
\[
\mathcal O_{\mathcal M}^{\even}
:=
\left\{
\sum_{\ell=1}^m
\left(
\sum_{j=1}^n\frac{\nu_j^{(\ell)}}{2m_j}-1
\right)
\in\mathcal O_{\mathcal M}
\,;\,
\sum_{\ell=1}^m\sum_{j=1}^n\nu_j^{(\ell)}
\ \text{is even}
\right\}.
\]
Then,
\[
\Phi_{j}(0,0)=0
\qquad
\text{whenever }
d_j\notin\mathcal O_{\mathcal M}^{\mathrm{even}}.
\]
\end{theorem}
Note that \eqref{eq:ex} gives the pointwise expansion at the origin $p$ as \begin{equation}\label{eq:pex}
B_k(p,p)\sim k^{\sum m_j^{-1}}(b_0+k^{-d_1}b_1+\cdots),    
\end{equation}where $b_0>0$.  

In complex dimension one, Marinescu--Savale \cite{marinescu2024bochner} assume that every point is of finite type. Note that the finite-type condition at a point implies good finite type with a decoupled model, and the curvature is strictly positive somewhere. Hence, $\deg L=\frac{1}{2\pi}\int_X\ii\Theta_h(L)>0$, and $L$ is ample. Therefore, the finite-type assumption need only be imposed at the point under consideration.

\begin{corollary}[Local version of the Marinescu--Savale result]\label{corr}
Let $X$ be a compact Riemann surface, and let $(L,h)$ be a semipositive line bundle. Suppose that $p$ is a finite-type point of $h$. Then, the off-diagonal decay \eqref{poff} and the full asymptotic expansions \eqref{eq:ex} and \eqref{eq:pex} hold for the Bergman kernel $B_k(x,y)$ of $L^k$ near $p$.
\end{corollary}

\subsection{Applications to generically branched coverings}
Let $f:X\to Y$ be a generically finite holomorphic map between compact
complex manifolds of dimension $n$, and let $L$ be an ample line bundle
on $Y$. The pullback $f^*L$ is nef and big, and its ample locus coincides with the locus where $f$ is locally finite (see Proposition~\ref{amplelocus}). Let
\[
\Ram(f):=\{x\in X:\rank_{\Cs}df_x<n\},
\]
be its ramification locus. Since $f$ is generically finite, $\Ram(f)$ is a (possibly singular) analytic set in $X$ with codimension one. Let $h$ be a positive Hermitian
metric, i.e. $\ii\Theta_{h}(L)>0$, on $L$. The pull-back metric
$f^*h$ has semipositive curvature with degenerate locus along $\Ram(f)$.
Assume $p\in\Ram(f)$ such that
\begin{equation}\label{rank_assumption}
\rank_{\Cs}df_{p}=n-1,\,\text{and }f\,\text{ is locally finite at }p.
\end{equation}The corank-one assumption is imposed mainly for simplicity; see Remark~\ref{rem:higher-corank} for possible extensions to stratified higher-corank settings. 

Let $m=m_{p}(f)$ be the local multiplicity of $f$ at $p$. Namely,
for a sufficiently small neighborhood $U$ of $p$,
\begin{equation}\label{m}
 \#\bigl(f^{-1}(y)\cap U\bigr)=m   
\end{equation}
for every regular value $y$ sufficiently close to $y_0=f(p)$.
Equivalently, $m$ is the number of local sheets meeting at $p$.

Fix a positive Hermitian $(1,1)$-form $\omega$ on $X$. We first establish the Bergman kernel
asymptotics of $(f^*L)^k$, endowed with the metric $(f^*h)^k$, near $p$. Since
$p\in\Ram(f)$, the local multiplicity $m$ defined in \eqref{m} satisfies $m\geq 2$, and we fix
the weight vector $\Lambda=(\frac{1}{2m},\frac12,\cdots,\frac12)$. We choose
\emph{normalized adapted coordinates} $z=(z_1,z')$, $z'=(z_2,\cdots,z_n)$, centered at $p$
(cf. Section~\ref{sec_b}). Namely, $z_j=w_j\circ f$ for $j=2,\dots,n$ for suitable holomorphic
coordinates $w=(w_1,w')$ centered at $f(p)$, the vector $\frac{\partial}{\partial z_1}$ spans
$\ker df_{p}$, and $\big|\frac{\partial}{\partial z_1}\big|_{\omega,p}=1$. These coordinates
need not be unitary with respect to $\omega$ at $p$. In these coordinates, the model weight
of the pulled-back metric $f^*h_L$ takes the form
\begin{equation*}
     \phimod(z)=\frac{1}{2}a_0(p)\,|z_1|^{2m}+\sum_{j=2}^{n}c_j|z_j|^2,
\qquad a_0(p),c_2,\cdots,c_n>0 .
\end{equation*}
Here, $a_0(p)>0$ is given in Definition~\ref{Def_a0} and is intrinsic (cf. Lemma~\ref{Lem_a0_intrinsic}). We then set
$S_k(z):=(k^{-\frac{1}{2m}}z_1,k^{-\frac12}z_2,\cdots,k^{-\frac12}z_n)$. 
\begin{theorem}\label{Thm_branched1}
Under the assumption \eqref{rank_assumption} and the construction above, there exist kernels
$\Phi_j\in\Cinf(\Cs^n\times\Cs^n)$, $j\geq1$, such that
\begin{equation}\label{eq:ex_branched}
B_k(S_k(z),S_k(w))
\sim
k^{n-1+\frac1m}
\left(
\Bmod(z,w)
+
\sum_{j=1}^{\infty}k^{-\frac{j}{2m}}\Phi_{j}(z,w)
\right),
\end{equation}
locally uniformly in the $\Cinf$-topology on $\Cs^n\times\Cs^n$. We have
\begin{equation}\label{PHI}
\Bmod(0,0)=\frac{a_0(p)^{\frac1m}}{(2\pi)^n\Gamma(1+\frac1m)}\,
\det_{\omega}\Big(\ii\Theta_{f^*h_L}(f^*L)_{p}\big|_{(\ker df_{p})^{\perp_\omega}}\Big),
\end{equation}
where the determinant is taken with respect to the Hermitian metric induced by $\omega$ on
$(\ker df_{p})^{\perp_\omega}$. Also, the off-diagonal statement holds in the sense of \eqref{poff} in Theorem~\ref{Thm_de}.
\end{theorem}
Next, we denote $M:=\operatorname{Ram}(f)$ and assume further that $M$ is a smooth hypersurface near $p$ and
\begin{equation}\label{12345}
    \rank_{\Cs} d(f\mid_M)_{p}=n-1 .
\end{equation}
After shrinking the polydisc, we may choose holomorphic coordinates $z=(z_1,z')$ on
$D=\Delta\times D'$ such that $M\cap D=\{z_1=0\}$ and $f(z_1,z')=(z_1^m,z')$
(see the proof of Proposition~\ref{prop.uniform}).
\begin{proposition}[Uniform diagonal asymptotics]
\label{prop.uniform}
Assume the hypotheses of Theorem~\ref{Thm_branched1} and \eqref{12345}. 
For any compact set $K\Subset M\cap D$, there exist continuous functions
$b_j\in \mathscr{C}^0(M\cap D)$, $j\in\mathbb N_0$,
such that
\[
B_k(p,p)\sim
k^{n-1+\frac1m}
\sum_{j=0}^{\infty}
k^{-\frac{j}{2m}}b_j(p),
\]
in the $\mathscr{C}^0$-topology on $K$. Here, $b_0(p)$ is given by the right-hand side of \eqref{PHI} at $p$.
\end{proposition}

\begin{example}[A double-plane K3 surface]
Let $s_6\in H^0(\Ps^2,\mathcal O_{\Ps^2}(6))$ have smooth zero locus
$\mathcal S$, and let
\[
X=\bigl\{(y,\xi):\xi^{\otimes2}=s_6(y)\bigr\}
\subset\operatorname{Tot}\mathcal O_{\Ps^2}(3).
\]
Then the projection $f:X\to\Ps^2$ is a double cover branched along
$\mathcal S$, and $X$ is a smooth K3 surface (cf. \cite[Example~1.3(iv)]{Huybrechts_2016}). Its ramification locus $M$ is smooth,
$f|_M:M\to\mathcal S$ is biholomorphic, and 
$\rank_{\Cs}d(f|_M)=1$.

Take $L=\mathcal O_{\Ps^2}(1)$ with the Fubini--Study metric $h_{\mathrm{FS}}$ and fix a
positive Hermitian form $\omega$ on $X$. Fix $p\in M$. Let $v_p\in T_p^{1,0}X$ be any $\omega$-unit vector, which is perpendicular to $\ker df_p$, and set
$c_2(p):=\left|df_p(v_p)\right|_{\Theta}^{2}$, where $|\cdot|_{\Theta}$ is the Hermitian norm on
$T_{f(p)}^{1,0}\Ps^2$ induced by the positive curvature form
$\ii\Theta_{h_{\mathrm{FS}}}(L)_{f(p)}$. We have
\[
    B_k(p,p)
    \sim
    k^{3/2}\sum_{j=0}^{\infty}b_j(p)k^{-j/4},
    \qquad
    b_0(p)
    =
    \frac{\sqrt{a_0(p)}\,c_2(p)}{2\pi^{5/2}},
\]uniformly on $M$.
\end{example}

\subsection{Certain non-quasihomogeneous models}

We conclude this section with a result of independent interest about
certain non-quasihomogeneous models, for which localization and
leading-order asymptotics can still be established. The example below
shows that the quasihomogeneity assumption in
Theorem~\ref{Thm_Full expansion} cannot in general be omitted if one expects an
asymptotic expansion purely in powers of the semiclassical parameter. Otherwise, 
logarithmic factors may occur. The corresponding abstract result is
stated in Theorem~\ref{Thm_non} of Section~\ref{sec_non}. Here, we only present
a concrete example.

\begin{example}[A non-quasihomogeneous projective model]
Let $X=\mathbb P^2$ and $L=\mathcal O_{\mathbb P^2}(1)$, and let $[Z_0:Z_1:Z_2]$ be the homogeneous coordinates on $\mathbb P^2$.
For a holomorphic linear form $S\in H^0(\mathbb P^2,L)$, define
\begin{equation*}
    |S|_h^2([Z])
    :=
    \frac{|S(Z)|^2}
    {\left(
        |Z_0|^6
        +3|Z_0Z_1Z_2|^2
        +3|Z_1|^6
        +3|Z_2|^6
    \right)^{1/3}}.
\end{equation*}
Since the expression in parentheses is positive on
$\mathbb C^3\setminus\{0\}$ and homogeneous of degree six, this defines
a smooth Hermitian metric on $L$. Moreover, its curvature is
semipositive. We now consider the affine chart
\[
    U_0=\{Z_0\neq 0\},
    \qquad
    z_j=\frac{Z_j}{Z_0},
    \quad j=1,2,
\]
and the local frame $e_L=Z_0$ of $L$. In this frame,
\[
    |e_L|_h^2
    =
    \left(
        1
        +3|z_1z_2|^2
        +3|z_1|^6
        +3|z_2|^6
    \right)^{-1/3}.
\]
At the point $p=[1:0:0]$, which corresponds to $z=0$, we have
\[
    \vphi(z)
    =-\frac{1}{2}\log |e_L|_h^2=
    \underbrace{\frac{1}{2}(|z_1z_2|^2
    +|z_1|^6
    +|z_2|^6)}_{:=\phimod}
    +O(|z|^8).
\]Also, we let $\omega_{\mathrm{FS}}=\ii\partial\db\log(|Z_0|^2+|Z_1|^2+|Z_2|^2)$ be the Fubini--Study metric on $\Ps^2$. In the chart $U_0$, this gives $\omega_{\mathrm{FS}}=\ii\partial\db\log(1+|z_1|^2+|z_2|^2)$, so that $\omega_{\mathrm{FS}}(p)=\ii\sum_{j=1}^2dz_j\wedge d\bar z_j$ and the volume form satisfies $\frac{\omega_{\mathrm{FS}}^2}{2!}=4\,dV$ at $p$, where $dV$ is the Lebesgue measure. We compute the Bergman kernel with respect to $(h,\omega_{\mathrm{FS}})$. Since $L=\mathcal O_{\mathbb P^2}(1)$ is ample, we have  $\operatorname{Amp}(L)=X$. Theorem \ref{Thm_non} gives the Bergman kernel asymptotics for $L^k$ as follows (cf. Subsection~\ref{subsec:example}): 
\begin{equation}\label{eq:non-qh-projective asymptotics}
    B_k(p,p)
    =
    \frac{3}{4\pi^2}
    \frac{k}{\log k}
    +
    O\left(\frac{k}{(\log k)^2}\right).
\end{equation}
\end{example}

\subsection*{Organization of the paper}
Section~\ref{sec22} states the general results under the two spectral gap assumptions, and Section~\ref{sec_2} fixes the notation. Section~\ref{sec_3} establishes off-diagonal estimates for the model operator on $\Cs^n$, and Section~\ref{sec_4} extends them to the rescaled and extended Laplacian. After adjusting the Hermitian form in Section~\ref{sec_adj}, we prove the localization and off-diagonal decay in Section~\ref{sec_5}, and the full expansion in the quasihomogeneous case in Section~\ref{sec_6}. Section~\ref{sec_de} verifies the spectral gap assumptions at good finite-type points and proves Theorem~\ref{Thm_de}. Section~\ref{sec_b} treats pullbacks under generically finite holomorphic maps, and Section~\ref{sec_non} treats non-quasihomogeneous models.
\subsection*{Acknowledgment} The author thanks his advisor, Dan Popovici, for his support and careful reading of the manuscript. The author is also grateful to his former advisor, Chin-Yu Hsiao, for carefully reading the manuscript, for advice on the analytic methods, and for his warm hospitality during the author's visit to Taiwan. The author thanks Shengxuan Zhou for many discussions related to this work. The author thanks Siarhei Finski for bringing several related works to the author's attention.

\subsection*{AI Disclosure}
The approach and main results of this paper were developed by the author without AI. AI tools were used to double-check the proofs and to improve the formulation, language and readability of the manuscript. The author reviewed and edited the text and takes full responsibility for the content of the manuscript.
\section{General spectral localization framework}\label{sec22}
In this section, we establish a more general spectral localization framework. We no longer assume that $X$ is compact or that $\ii\Theta_h(L)\geq 0$ globally. Instead, our results are formulated under two spectral gap assumptions. We first recall the following notion introduced by Hsiao and Marinescu \cite{hsiao2014asymptotics}. 
\begin{definition}\label{spectral gap}
Let $D\subset X$ be open. We say that $\{\Box_k\}_{k\in\N}$ has a \textbf{localized mild spectral gap} on $D$ if there exist constants $C_D>0$ and $k_0,N_0\in\N$ such that
\[
\|(\Id-B_k)u\|_{h^k,\omega}
\leq
C_D k^{N_0}\|\Box_k u\|_{h^k,\omega},
\]
for all $k\geq k_0$ and all $u\in C_c^\infty(D,L^k)$.
\end{definition}

Let $\Lambda=(\lambda_1,\cdots,\lambda_n)$ be positive rational numbers with $\lambda_1\leq \lambda_2\leq \cdots \leq \lambda_n$. We call $\Lambda$ a weight vector if $\sum_{j=1}^n\lambda_j n_j=1$ for some $n_1,\cdots,n_n\in\N_0$. The size of $\Lambda$ is $|\Lambda|:=\sum_{j=1}^n\lambda_j$. For $\alpha,\beta\in\N^n_0$, the $\Lambda$-weighted degree of $(\alpha,\beta)$ is \[
|(\alpha,\beta)|_\Lambda:=\sum_{j=1}^n\lambda_j(\alpha_j+\beta_j).
\] The local weight of $h^k$ is $k\varphi$. If we rescale the trivialization \eqref{eq} by $\Lambda$, we get
\[
k\,\varphi\big(k^{-\lambda_1}z_1,\dots,k^{-\lambda_n}z_n\big)
\sim\sum_{\alpha,\beta}k^{1-|(\alpha,\beta)|_\Lambda}a_{\alpha,\beta}z^\alpha\bar z^\beta .
\]
Here, the terms of weighted degree $<1$ grow with $k$, those of degree $1$ stay constant, and those of
degree $>1$ tend to zero.

Let $\psi$ be a real polynomial on $\Cs^n$. We write $\Box_\psi:=\db_\psi^*\db$ for the Kodaira
Laplacian on functions (cf.\ \eqref{set:Box}). Here $\db_\psi^*$ is the formal adjoint of $\db$
with respect to the standard Hermitian $(1,1)$-form on $\Cs^n$ and the inner product given by the weighted measure $e^{-2\psi}dV$, where
$dV$ is the standard Lebesgue measure. 

\begin{definition}[Model gap hypothesis]\label{def:model-gap}
Let $\Lambda$ be a weight vector, and let $a_{\alpha,\beta}$ be the Taylor coefficients \eqref{eq} of the local weight $\varphi$ with respect to local holomorphic coordinates $z$ centered at $p$ and a local holomorphic frame $e_L$ of $L$.
\begin{enumerate}
\item[(i)] We say that $p\in X$ satisfies the \textbf{model gap condition} with weight vector
$\Lambda$ if there exist $k_0\in\N$ and $c>0$ such that the truncated rescaled weights
\begin{equation*}
\phimodk(z):=\sum_{|(\alpha,\beta)|_\Lambda\leq 1}k^{1-|(\alpha,\beta)|_\Lambda}a_{\alpha,\beta}z^\alpha\bar z^\beta
\end{equation*}
satisfy $\Spec\Box_{\phimodk}\cap(0,c)=\varnothing$ for all $k\geq k_0$. Such coordinates are called \textbf{model coordinates}.
\item[(ii)] More strongly, we say that $p$ satisfies the \textbf{quasihomogeneous model gap condition} if (i) holds and $a_{\alpha,\beta}=0$ whenever $|(\alpha,\beta)|_\Lambda<1$. In this case, $\phimodk$ is independent of $k$, and we denote
\begin{equation*}
\phimod(z):=\sum_{|(\alpha,\beta)|_\Lambda=1}a_{\alpha,\beta}z^\alpha\bar z^\beta .
\end{equation*}
\end{enumerate}
\end{definition}
We will show that good finite type (resp. quasihomogeneous good finite type) implies the model gap assumption (resp. quasihomogeneous model gap) in Section \ref{sec_de}. For model coordinate $z=(z_1,\dots,z_n):D\to\Cs^n$ with weight vector
$\Lambda=(\lambda_1,\dots,\lambda_n)$, define the scaling map
$\Sk(z):=\big(k^{-\lambda_1}z_1,\dots,k^{-\lambda_n}z_n\big)$. 
\begin{theorem}[Localization and off-diagonal decay]\label{Thm_Loc_Off}
Assume that $\Box_k$ has a localized mild spectral gap on an open set $D\ni p$, and that $p$
satisfies the model gap condition with weight vector $\Lambda=(\lambda_1,\dots,\lambda_n)$ in the
model coordinates $z=(z_1,\dots,z_n):D\to\Cs^n$. Then, the Bergman kernels for $L^k$ have the following asymptotic results.

\smallskip
\noindent\textbf{(a) Localization.} For every
$N,r\in\N$ and every compact set $K\Subset\Cs^n$ there is $C=C(N,r,K)>0$ with
\begin{equation*}
\Big\|k^{-2|\Lambda|}B_k\big(\Sk(z),\Sk(w)\big)-\tBk(z,w)\Big\|_{\mathscr{C}^r(K\times K)}\leq C\,k^{-N}.
\end{equation*}
 Here, $\tBk(z,w)\in\Cinf(\Cs^n\times\Cs^n)$ is the Bergman kernel on $\Cs^n$, which depends only on the germs of $h$ and $\omega$ at $p$. This is defined in \eqref{tBk}. Here, $\|\cdot\|_{\Cr(K\times K)}$ is  rescaled weighted $C^r$-norm induced by \eqref{Crs}.  
\smallskip

\noindent\textbf{(b) Off-diagonal decay.} For  $\delta>0$, denote
\[
D_k:=\{x\in D;|z_j(x)|<  k^{-\lambda_j}\}\quad;\quad D_{k,\delta}:=\{y\in D ;
|z_j(y)|< k^{-\lambda_j+\delta}\}.
\]Note that $D_k\Subset D_{k,\delta}\subset D$ for large enough $k\in\N$. Then, for any $E\Subset X$ and $\delta>0$, 
\begin{equation}\label{off}
|B_k(x,y)|_{\mathscr{C}^r}=O(k^{-\infty}),\quad  \text{uniformly for }x\in D_k, y\in E\setminus D_{k,\delta}.  
\end{equation}
 The same bound holds with $x$ and $y$ interchanged.
\end{theorem}

\begin{definition}[Orders of the expansion]\label{order}
Let $\Lambda=(\lambda_1,\dots,\lambda_n)$ be a weight vector. For $m\in\N_0$ and
$\mathbf n=\big(n^{(1)},\dots,n^{(m)}\big)\in(\N_0^n)^m$, let
\[
d(\mathbf n):=\sum_{i=1}^m\Big(\sum_{j=1}^n n^{(i)}_j\lambda_j-1\Big),
\qquad
\ell(\mathbf n):=\sum_{i=1}^m\sum_{j=1}^n n^{(i)}_j .
\]
We call $\mathbf n$
\emph{admissible} if $\sum_{j=1}^n n^{(i)}_j\lambda_j\geq 1$ for every $i=1,\dots,m$. Let
\[
\mathcal{O}_\Lambda:=\big\{d(\mathbf n)\ ;\ \mathbf n\ \text{admissible}\big\},\qquad
\mathcal{O}^{\even}_\Lambda:=\big\{d(\mathbf n)\ ;\ \mathbf n\ \text{admissible},\ \ell(\mathbf n)\ \text{even}\big\}.
\]
We list the elements of $\mathcal{O}_\Lambda$ in increasing order $0<d_1<d_2<d_3<\cdots$.

\end{definition}

\begin{theorem}[Full expansion]\label{Thm_Full expansion}
Assume that $\Box_k$ has a localized mild spectral gap on an open set $D\ni p$, and that $p$
satisfies the quasihomogeneous model gap condition with weight vector
$\Lambda=(\lambda_1,\dots,\lambda_n)$ in the model coordinates
$z=(z_1,\dots,z_n):D\to\Cs^n$. 
Then, there are kernels $\Phi_{j}\in\Cinf(\Cs^n\times\Cs^n)$, $j\geq1$, such that for all
$N,r\in\N$ and compact set $K\Subset \Cs^n$, there is $C=C(N,r,K)>0$ satisfying
\begin{equation}\label{full_expansion}
\Big\|B_k\big(\Sk(z),\Sk(w)\big)-k^{2|\Lambda|}\Big(\Bmod(z,w)
+\sum_{j=1}^{N}k^{-d_j}\Phi_{j}(z,w)\Big)\Big\|_{\Cr(K\times K)}\leq C\,k^{2|\Lambda|-d_{N+1}},
\end{equation}for sufficiently large $k$. Here $\Bmod$ is the Bergman kernel of the Fock space on
$\Cs^n$ with weight measure $e^{-2\phimod(z)}\omega_0^n/n!$, where $\omega_0$ is the constant Hermitian $(1,1)$-form on $\Cs^n$ with $\omega_0^n/n!=dV_\omega(p)$ (cf.\ Section~\ref{sec_3}). Moreover, if the model is \emph{even} in the sense that $\phimod(-z)=\phimod(z)$, we have  $\Phi_{j}(0,0)=0$ for all $d_j\notin O^{\even}_\Lambda$ .
\end{theorem}
\begin{remark}
  All the results of this section extend naturally to $L^k\otimes E$, where $E$ is a fixed Hermitian holomorphic vector bundle, under the corresponding spectral gap assumptions.
\end{remark}

\section{Preliminaries, notation and set-up}\label{sec_2}
\subsection{Standard notation}

We first list some standard notation.

\begin{itemize}
    \item We write $\ii:=\sqrt{-1}$, $\N_0:=\{0\}\cup\N$, $\N_0^n:=\bigl\{
            \alpha=(\alpha_1,\ldots,\alpha_n):
            \alpha_j\in\N_0
        \bigr\}$. For $\alpha\in\N_0^n$, we set
    $|\alpha|:=\alpha_1+\cdots+\alpha_n$.

    \item We use $x,y$ for points on manifolds, and $z,w,\xi$ for
    variables in local holomorphic coordinates or on $\Cs^n$. The
    letter $\zeta$ denotes the spectral variable; see \eqref{set:res}. We denote $B(r)\subset\Cs^n$ as $\{|z|<r\}$.

    \item For $\alpha\in\N_0^n$, we write
    \[
        \partial_z^\alpha
        :=
        \left(\frac{\partial}{\partial z_1}\right)^{\alpha_1}
        \cdots
        \left(\frac{\partial}{\partial z_n}\right)^{\alpha_n},
        \qquad
        \partial_{\bar z}^\alpha
        :=
        \left(\frac{\partial}{\partial\bar z_1}\right)^{\alpha_1}
        \cdots
        \left(\frac{\partial}{\partial\bar z_n}\right)^{\alpha_n}.
    \]

    \item Let $X$ be a smooth manifold and let $E\to X$ be a smooth
    vector bundle. We denote by $\Cinf(X,E)$ the space of smooth
    sections of $E$, and by $\Cinf_c(X,E)$ the subspace of smooth sections with compact support. If $E=X\times\Cs$ is the trivial line
    bundle, we write $\Cinf(X)$ and $\Cinf_c(X)$.

    \item In fixed local coordinates near $z_0$, for
    $f\in\Cinf(X)$ and $r\in\N_0$, we use the notation
    \[
        |f(z_0)|_{\mathscr C^r}
        :=
        \max_{\substack{\alpha,\beta\in\N_0^n\\
        |\alpha|+|\beta|\leq r}}
        \left|
            \bigl(\partial_z^\alpha
            \partial_{\bar z}^\beta f\bigr)(z_0)
        \right|.
    \]We denote $\|f\|_{\Cr(U)}=\sup_{z\in U}|f|_{\Cr}$ for some bounded domain $U$.
    \item We denote by $\omega_e:=\frac{\ii}{2}\sum_{j=1}^n dz_j\wedge d\bar z_j$ the standard Hermitian $(1,1)$-form on $\Cs^n$, and by $dV:=\frac{\omega_e^n}{n!}=\left(\frac{\ii}{2}\right)^n
        (dz_1\wedge d\bar z_1)\wedge\cdots\wedge
        (dz_n\wedge d\bar z_n)$
    the standard Lebesgue measure on $\Cs^n$.
\end{itemize}
\subsection{Setup and preliminaries}

Let $X$ be a complex manifold of complex dimension $n$, and let $\omega$ be a positive $(1,1)$-form, which induces a Hermitian metric $\la\cdot|\cdot\ra_{\omega}$ on the holomorphic vector bundle $\Lambda^{1,0}TX$. The volume form with respect to $\omega$ is given by $dV_\omega=\omega^n/n!$. Let $D$ be a local chart centered at $p\in X$, and let $z=(z_1,\cdots,z_n)$ be holomorphic coordinates. Let $\Lambda^{0,q}T^*X$ denote the vector bundle of $(0,q)$-forms.

Let $L\to X$ be a holomorphic line bundle, and let $h$ be a smooth Hermitian metric on $L$. Given a local holomorphic frame $e_L:D\subset X\to L$, we define the associated local weight function of $h$ by $|e_L(z)|_h^2=e^{-2\varphi(z)}$, where $\varphi\in\Cinf(D,\R)$. Let $\ii\Theta_h(L)$ be the Chern curvature of $L$, which is locally given by $\ii\Theta_h(L)=2\ii\partial\bar\partial\varphi$. Note that $h$ and $\la\cdot|\cdot\ra_\omega$ induce a Hermitian metric on the vector bundle ${\Lambda^{0,q}T^*X}\otimes L$, with $q\in\{0,1,\cdots,n\}$, denoted by $\la\cdot|\cdot\ra_{h,\omega}$. The $L^2$-Hermitian inner products on the spaces $\Cinf_c(X,{\Lambda^{0,q}T^*X})$ and $\Cinf_c(X,{\Lambda^{0,q}T^*X}\otimes L)$ are given by
\begin{equation}\label{set:inner_product}
    \begin{split}
        (u|v)_{\omega}
        &=
        \int_X\la u(\cdot)|v(\cdot)\ra_{\omega}\,dV_\omega,
        \quad
        u,v\in\Cinf_c(X,{\Lambda^{0,q}T^*X});
        \\
        (u|v)_{h,\omega}
        &=
        \int_X\la u(\cdot)|v(\cdot)\ra_{h,\omega}\,dV_\omega,
        \quad
        u,v\in\Cinf_c(X,{\Lambda^{0,q}T^*X}\otimes L).
    \end{split}
\end{equation}
We denote the corresponding norms by $\|u\|^2_{h,\omega}:=(u|u)_{h,\omega}$ and $\|v\|_{\omega}^2:=(v|v)_{\omega}$. Let $L^2_{\omega}(X,{\Lambda^{0,q}T^*X})$ and $L^2_{h,\omega}(X,{\Lambda^{0,q}T^*X}\otimes L)$ denote the completions of $\Cinf_c(X,{\Lambda^{0,q}T^*X})$ and $\Cinf_c(X,{\Lambda^{0,q}T^*X}\otimes L)$, respectively.


Let
$$
\db^*_{h,\omega}:
\Cinf(X,\Lambda^{0,1}T^*X\otimes L)
\longrightarrow
\Cinf(X,L)
$$
be the formal adjoint of the $L$-valued $\db$-operator
$$
\db:\Cinf(X,L)\longrightarrow
\Cinf(X,\Lambda^{0,1}T^*X\otimes L)
$$
with respect to the inner product $(\cdot|\cdot)_{h,\omega}$. We denote the Kodaira Laplacian, or the $L$-valued $\db$-Laplacian, by
$$
\Box_{h,\omega}
:=
\db^*_{h,\omega}\db:
\Cinf(X,L)\longrightarrow\Cinf(X,L).
$$
We consider the Gaffney extension (cf.~\cite{gaffney1955hilbert})
\begin{equation}\label{set:Box}
  \Box_{h,\omega}:\Dom\Box_{h,\omega}\longrightarrow L^2_{h,\omega}(X,L),
\end{equation}
which is a self-adjoint operator. For an explicit description of \(\Dom\Box_{h,\omega}\), we refer the readers to, e.g., \cite[\S~2.3]{hsiao2014asymptotics}. If $(X,\omega)$ is complete, the operator $\Box_{h,\omega}$ initially defined on $\Cinf_c(X,L)$ is essentially self-adjoint; see \cite[Corollary~3.3.4]{ma2007holomorphic}. For any $\zeta\in\Cs\setminus\R$, we consider the resolvent of $\Box_{h,\omega}$ at $\zeta$, denoted by
\begin{equation}\label{set:res}
(\zeta-\Box_{h,\omega})^{-1}:
L^2_{h,\omega}(X,L)
\longrightarrow
\Dom\Box_{h,\omega}
\subset
L^2_{h,\omega}(X,L),
\end{equation}
which is a bounded operator satisfying (see \cite{davies1995spectral})
\[
\|(\zeta-\Box_{h,\omega})^{-1}\|_{L^2\to L^2}
\leq
|\operatorname{Im}\zeta|^{-1}.
\]
Here, $\|\cdot\|_{L^2\to L^2}$ denotes the operator norm.


We now introduce the Helffer--Sj\"ostrand formula \cite{helffer2005equation}. Let $f\in\Cinf_c(\R)$, and let $\widetilde f\in\Cinf_c(\Cs)$ be an almost analytic extension of $f$, namely, $\widetilde f|_{\R}=f$, and for every $N\in\N$, there exists a constant $C_N>0$ such that
\begin{equation}\label{eq:HS_aa}
    \left|
    \frac{\partial\widetilde f}{\partial\bar\zeta}(\zeta)
    \right|
    \leq
    C_N|\operatorname{Im}\zeta|^N,
    \quad
    \text{for }\zeta\in\Cs.
\end{equation}
Then, the functional calculus $f(\Box_{h,\omega})$ of $\Box_{h,\omega}$ is given by the Helffer--Sj\"ostrand formula:
\begin{equation}\label{HS}
    f(\Box_{h,\omega})
    =
    \frac{1}{2\pi\ii}
    \int_{\Cs}
    (\zeta-\Box_{h,\omega})^{-1}\,
    d\zeta\wedge\db\widetilde f,
\end{equation}
where $\db\widetilde f=\frac{\partial\widetilde f}{\partial\bar\zeta}\,d\bar\zeta$ is a $(0,1)$-form on $\Cs$.

We denote the Bergman projection by
$$
B_{h,\omega}:
L^2_{h,\omega}(X,L)
\longrightarrow
\operatorname{Ker}\db\cap L^2_{h,\omega}(X,L)=\Ker\Box_{h,\omega}.
$$
Let $B_{h,\omega}(x,y)\in\Cinf(X\times X,L\boxtimes L^*)$ denote the Bergman kernel, where $L\boxtimes  L^*$ is a vector bundle over $X\times X$ whose fiber at $(x,y)\in X\times X$ is $L_x\otimes L_y^*$. It satisfies
\begin{equation}\label{Bh}
    B_{h,\omega}u(x)
    =
    \int_X
     B_{h,\omega}(x,y)u(y)
    \frac{\omega^n(y)}{n!},
\end{equation}
where $ B_{h,\omega}(x,y)u(y)$ is the fiberwise pairing of $L_y$ and $L_y^*$. 

In the case $X=\Cs^n$, let $L\simeq\Cs^n\times\Cs$ be the trivial line bundle. We fix the unit section $\mathbf 1$ as a frame, with $\mathbf 1(z)=(z,1)$, and assume $\la\mathbf 1|\mathbf 1\ra_h=e^{-2\varphi}$. In this case, we write $(\cdot|\cdot)_{\varphi,\omega}:=(\cdot|\cdot)_{h,\omega}$ and $\|\cdot\|_{\varphi,\omega}:=\|\cdot\|_{h,\omega}$. We also set
$$
L^2_{\varphi,\omega}
\bigl(\Cs^n,\Lambda^{0,q}T^*\Cs^n\bigr)
:=
L^2_{h,\omega}
\bigl(\Cs^n,\Lambda^{0,q}T^*\Cs^n\otimes L\bigr).
$$
When $q=0$, we simply write $L^2_{\varphi,\omega}(\Cs^n)$. This is equivalent to considering the weighted $L^2$-space on $\Cs^n$ with inner product
\begin{equation}\label{set:inner_product2}
(f|g)_{\varphi,\omega}
=
\int_{\Cs^n}
f\overline g\,e^{-2\varphi}\,dV_{\omega},
\quad
f,g\in L^2_{\varphi,\omega}(\Cs^n).
\end{equation}
We also denote $\db^*_{\varphi,\omega}:=\db^*_{h,\omega}$, $\Box_{\varphi,\omega}:=\Box_{h,\omega}$, and $B_{\varphi,\omega}:=B_{h,\omega}$. Note that
\begin{equation}\label{Bk2}
B_{\varphi,\omega}u(z)
=
\int_{\Cs^n}
B_{\varphi,\omega}(z,w)u(w)\,dV_{\omega}(w).
\end{equation}
Let $g=f\cdot e_L^k:D\subset X \to L^k$ be a local section, where $f\in\Cinf(D)$, the pointwise $\Cr$-norm at $z_0\in D$ is given by \begin{equation}\label{Cr-norm}
 |g(z_0)|^2_{\Cr}:=|f(z_0)|^2_{\Cr}|e_L^k|^2_{h^k}=|f(z_0)|^2_{\Cr} e^{-2k\vphi}.   
\end{equation} Also, $\|g\|_{\Cr(D)}^2:=\sup_{z\in D}|g(z)|^2_{\Cr}$. If $g\in\Cinf_c(D,L^k)$, the Sobolev $r$-norm for $r\in\N$ is given by \begin{equation}\label{Sobolevnorm}
  \|g\|^2_{H^{r}}=\sum_{|\alpha|+|\beta|\leq r}\int_D|\partial^\alpha_z\partial^\beta_{\bar z}f|^2e^{-2k\vphi}dV_\omega.  
\end{equation} The Sobolev space $H^r(D,L^k)$ is given by the completion of $\Cinf_c(D,L^k)$ with respect to $\|\cdot\|_{H^r}$. It is standard that such definition can be extended to $r\in\R$. Also, the norms extend to the sections of $L^k\boxtimes L^{-k}$.  For example, if we write Bergman kernel locally as $B_k(z,w)=\mathbf{B}_{k}(z,w)e_L^k(z)\otimes e_L^{-k}(w)$, then \begin{equation}\label{Crs}
 |B_k(z,w)|^2_{\Cr}=|\mathbf{B}_{k}(z,w)|^2_{\Cr}\cdot e^{-2k(\vphi(z)-\vphi(w))}
.   \end{equation}

We recall the notion of the ample locus of a big line bundle (see \cite[Definition 1.3]{boucksom2010monge}).
A K\"ahler current is a closed real $(1,1)$-current $T$
satisfying $T\geq\varepsilon\omega$ for some $\varepsilon>0$.

\begin{definition}\label{ample}
Let $L$ be a big line bundle over a compact complex manifold $X$.
A point $p\in X$ is said to lie in the \textbf{ample locus}
$\operatorname{Amp}(L)$ if there exists a K\"ahler current
$T\in c_1(L)$ with analytic singularities which is smooth
in a neighbourhood of $p$.
\end{definition}

\section{Model operator and off-diagonal estimate}\label{sec_3}
In this section, we study the resolvents of the Kodaira Laplacian on $\Cs^n$. We work with the weight function $\phimodk$ (see Definition \ref{def:model-gap}) on the trivial line bundle $\Cs^n\times \Cs$ over $\Cs^n$. Let $v(0):=\frac{dV_\omega}{dV}(p)$ be the volume density of $\omega$ at $p$ in the model coordinates. We define the flat Hermitian $(1,1)$-form
\begin{equation*}
\omega_0(z):=v(0)^{\frac1n}\,\omega_e(z),\quad z\in \Cs^n,
\end{equation*}
which has the same volume form as $\omega$ at $p$, i.e.\ $\omega_0^n/n!=v(0)\,dV$. Since $\omega_0$ is a constant multiple of $\omega_e$, we have $\Box_{\phimodk,\omega_0}=v(0)^{-\frac1n}\Box_{\phimodk,\omega_e}$ and $B_{\phimodk,\omega_0}(z,w)=v(0)^{-1}B_{\phimodk,\omega_e}(z,w)$. In particular, the spectral gap in Definition~\ref{def:model-gap} holds for $\omega_0$ if and only if it holds for $\omega_e$.

Recall that $L^2_{\omega_0}(\Cs^n)$ and $L^2_{\phimodk,\omega_0}(\Cs^n)$ are the Hilbert spaces induced by $(\cdot|\cdot)_{\omega_0}$ and $(\cdot|\cdot)_{\phimodk,\omega_0}$, respectively (see \eqref{set:inner_product},\eqref{set:inner_product2}). Denote the Gaffney extension of the Kodaira Laplacian as \begin{equation*}
\Boxmodk:=\Box_{\phimodk,\omega_0}=\db^*_{\phimodk,\omega_0}\db:\Dom\Boxmodk\subset L^2_{\phimodk,\omega_0}\to L^2_{\phimodk,\omega_0}.   
 \end{equation*}
To compute the connection more symmetrically, we use the unitary identification
\begin{equation}\label{eq:uni_model}
U_k:L^2_{\phimodk,\omega_0}(\Cs^n,\Lambda^{(0,q)}T^*\Cs^n)\longrightarrow L^2_{\omega_0}(\Cs^n,\Lambda^{(0,q)}T^*\Cs^n),
\qquad
f\longmapsto e^{-\phimodk}f,
\end{equation}where $q=0,1$.
Under this identification, we transform the $\db$-operator as shown in the commutative diagram:
\begin{equation*}
    \begin{tikzcd}[column sep=huge, row sep=large]
\Dom(\db)\subset L^2_{\phimodk,\omega_0}(\Cs^n)
\arrow[r,"\db"]
\arrow[d,"U_k"',"\simeq"]
&
L^2_{\phimodk,\omega_0}
(\Cs^n,\Lambda^{0,1}T^*\Cs^n)
\arrow[d,"U_k","\simeq"']
\\
\Dom(\Dk)\subset L^2_{\omega_0}(\Cs^n)
\arrow[r,"\Dk"]
&
L^2_{\omega_0}
(\Cs^n,\Lambda^{0,1}T^*\Cs^n),
\end{tikzcd}
\end{equation*}
where
$\Dk
:=U_k\db U_k^{-1}
=\db+(\db\phimodk)\wedge\cdot $. The unitary relation gives more operators with domain in $L^2_{\omega_0}(\Cs^n)$ as follows: \begin{equation*} \begin{split} \db^*_{\phimodk,\omega_0}\,&\longleftrightarrow \,\Dk^*:=e^{-\phimodk}\db^*_{\phimodk,\omega_0} e^{\phimodk}=\db^*_{\omega_0}+(\db\phimodk \,\wedge\cdot )^* ;\\ \Boxmodk\,&\longleftrightarrow\, \Hmodk:=e^{-\phimodk}\Boxmodk e^{\phimodk}=\Dk^*\,\Dk. \end{split} \end{equation*}
Here, $\db^*_{\omega_0}$ denotes the formal adjoint of $\db$ with respect
to the inner product $(\cdot|\cdot)_{\omega_0}$. We consider the Gaffney extension of $\Hmodk$,
\begin{equation*}
\Hmodk:\Dom(\Hmodk)\subset L^2_{\omega_0}(\Cs^n)
\longrightarrow L^2_{\omega_0}(\Cs^n).
\end{equation*}
This extension is self-adjoint. Therefore, for any $\zeta\in\Cs\setminus\R$, we define the resolvent\[\Rk:=(\zeta-\mathcal{H}_k)^{-1}:L^2_{\omega_0}(\Cs^n)\to \Dom\Hk.\]By the unitary relation, this equals $e^{-\phimodk}(\zeta-\Boxmodk)^{-1}e^{\phimodk}$. Let $\Rk(z,w)$ be the Schwartz kernel of $\Rk$. This means for any $u\in L^2_{\omega_0}$, we have\[\Rk u(z) = \int_{\Cs^n}\Rk(z,w)u(w)dV_{\omega_0}(w).\]
Since $\mathcal{H}_k$ is an elliptic operator, it is a standard fact that $\Rk(z,w)$ is smooth when $z \neq w$. In the rest of this section, our main goal is to estimate the $\mathscr{C}^r$-norms of this kernel as $|z-w|\to \infty$.

For $f,g\in\Cinf(\Cs^n)$ and
$u\in\Cinf(\Cs^n,\Lambda^{0,1}T^*\Cs^n)$, the
Leibniz identities give:
\begin{equation}\label{eq:LB}
\Dk(fg)
=
(\db f)\,g+f\,\Dk g,
\qquad
\Dk^*(fu)
=
\underbrace{[\Lambda_{\omega_0},-\ii\partial f\wedge\cdot]u}_{:= V_f u}
+
f\,\Dk^*u.
\end{equation}
Here, $\Lambda_{\omega_0}$ denotes the formal adjoint of the Lefschetz operator
$L_{\omega_0}:=\omega_0\wedge\cdot$; see, e.g.,
\cite[Lemma~2.7]{popovici2025twisted} for the second identity.

\begin{theorem}[Off-diagonal estimate]\label{Thm1}For any $N,r\in\N$ and a bounded domain $U\Subset\Cs$, there exist $\ell=\ell(r,U),M=M(N)\in\N$ and $C=C(N,r,U)>0$ such that  \begin{equation*}
    \Big| \Rk(z,w)\Big|_{\mathscr{C}^r}\leq \dfrac{C}{|\operatorname{Im }\zeta|^{M}}\cdot \dfrac{(1+|z|+|w|+k)^{\ell}}{|z-w|^N},
\end{equation*}for all $|z-w|\geq 1$, $k\in\N$ and $\zeta\in U\setminus\R$. In particular, we can omit the factor $k$ in the quasihomogeneous case $\phimodk=\phimod$. 
\end{theorem}
\begin{proof}
    Let $\chi\in\Cinf_c(\Cs^n,[0,1])$ be a cut-off function with $\chi=1$ for $|z|<\frac{1}{8}$ and $\chi=0$ for $|z|\geq \frac{1}{4}$. For any $z,w\in\Cs^n$ with $|z-w|\geq 1$, we denote $\chi_z:=\chi(\bullet-z)$ and $\chi_w:=\chi(\bullet-w)$. 
    
    Note that the coefficients of $\Hmodk$, together with all their derivatives,
have polynomial growth in the spatial variables and in $k$. Namely,  $\Cr$-norms of the coefficients of $\Hk$ on the ball $\{w\in\Cs^n;|w-z|<1\}$ are bounded by
$C_r(1+|z|+k)^{M_r}$, for some $C_r,M_r>0$. Moreover, the principal part of $\Hmodk$ is
uniformly elliptic. Hence, the standard interior elliptic estimate,
which is applied to $(\zeta-\Hmodk)^m$, is uniform for $\zeta$ in the bounded set $U$ up to a polynomial loss in $|z|+k$. More precisely, for every $m\in\N$ , there exist $\ell=\ell(m,U)\in\N$ and
$C=C(m,U)>0$ such that
 \begin{equation}\label{e1}
    \|\chi_zu\|_{H^{2m}} \leq C (1+|z|+k)^{\ell}\, \left(\|\widetilde{\chi}_z u\|_{L^2}+\|\widetilde{\chi}_z(\zeta-\Hmodk)^m u\|_{L^2}\right),
 \end{equation}where $\|\cdot\|_{H^{2m}}$ is the standard Sobolev norm. Here, $\widetilde{\chi}_z$ is another cut-off function such that $\supp\chi_z\subset\{\widetilde{\chi}_z=1\}$ and $\supp\widetilde{\chi}_z\subset \{\xi;|\xi-z|<1/2\}$. 
 
  Let $\delta_z$ denote the Dirac measure centered at $z$. Fix $r\in\mathbb N$ and choose $\mu>n+r$. The Sobolev embedding gives $H^\mu(\mathbb C^n)\hookrightarrow \mathscr C^r(\mathbb C^n)$.  Moreover, for $|\beta|\leq r$, the distributions
$\partial^\beta\delta_z$ belong to $H^{-\mu}(\mathbb C^n)$ with $\|\partial^\beta\delta_z\|_{H^{-\mu}}$
independent of $z$. Therefore, there are some $C,\widetilde{C}>0$ such that 
\begin{equation*}
|\Rk(z,w)|_{\mathscr C^r}
\leq C \sum_{|\alpha|+|\beta|\leq r} |\left((\partial^\alpha\delta_z)|\Rk (\partial^\beta\delta_w)\right)_{\omega_0}|\leq 
\widetilde C\|\chi_z\Rk\chi_w\|_{H^{-\mu}\to H^\mu},\end{equation*}
where $\|\cdot\|_{H^{-\mu}\rightarrow H^{\mu}}$ is the operator norm between the Sobolev spaces $H^{-\mu}$ and $H^{\mu }$. Since $\supp\Tilde{\chi}_z\cap\supp\chi_w=\varnothing$, we have  $\Tilde\chi_z(\zeta-\Hmodk)^m\Rk\chi_w\equiv0$ for any $m\in\N$. Combining this fact with \eqref{e1}, we estimate $\|\chi_z\Rk\chi_w\|_{H^{-\mu}\to H^\mu}$ as follows:   \[
\|\chi_z{\Rk}\chi_w\|_{H^{-\mu}\to H^\mu}\leq C(1+|z|+k)^{\ell(\mu,U)}\|\widetilde{\chi}_z{\Rk}\chi_w \|_{H^{-\mu}\to L^2}.\]Applying the same argument to the adjoint operator $\bar\zeta-\Hmodk$, or equivalently by duality in $w$-variable, we get \[
\|\widetilde\chi_z\Rk  \chi_w\|_{H^{-\mu}\rightarrow L^2}\leq C(1+|w|+k)^{\ell(\mu,U
)}\|\widetilde{\chi}_z{\Rk}
\widetilde{\chi}_w\|_{L^2\rightarrow L^2}.\]Therefore, for $|z-w|\geq 1$, \begin{equation}\label{e2}
 \lvert\Rk(z,w)\rvert_{\Cr}\leq  C(1+|z|+|w|+k)^{\ell(r,U)}\|\widetilde{\chi}_z{\Rk}\widetilde{\chi}_w\|_{L^2\rightarrow L^2}.\end{equation}
It remains to estimate $\|\widetilde\chi_z\Rk\widetilde\chi_w\|_{L^2\to L^2}$. Without loss of generality, it is sufficient to prove the following.
\begin{claim}For any $N\in\N$, there exists a constant $C=C(N,U)>0$ independent of $z$ and $w$ such that
 \begin{equation*}
 \|{\chi}_z{\Rk}{\chi}_w\|^2_{L^2\rightarrow L^2}\leq C|\operatorname{Im} \zeta|^{-N-2}\cdot|z-w|^{-N},
 \end{equation*}for all $\zeta\in U\setminus\R$ and $|z-w|\geq 1$.
\end{claim}

We first construct a nested sequence of cut-off functions.
Choose $\chi_1\in C_c^\infty(\Cs^n,[0,1])$ such that
$\chi_1=1$ on $B(1/3)$ and $\operatorname{supp}\chi_1\subset B(3/8)$. For $j\geq 1$, choose
$\chi_{j+1}$ such that $\operatorname{supp}\chi_j
\subset\{\chi_{j+1}=1\}$ and $\operatorname{supp}\chi_j\subset B(1/2)$ for all $j\in\N$. Denote \[
\chi^{(j)}_{z,w}:=\chi_j(\dfrac{\bullet-z}{|z-w|}).
\]Observe that $\supp \chi^{(j)}_{z,w}\cap \supp \chi_w=\varnothing$, and $\sup|\nabla\chi^{(j)}_{z,w}|\leq \sup|\nabla\chi_j|\cdot|z-w|^{-1}$ for all $j\in\N$. Since $\supp\chi_z\subset\{\chi_{z,w}^{(1)}=1\}$, we have  $\|\chi_z\Rk\chi_w\|_{L^2\To L^2}\leq \|\chi_{z,w}^{(1)}\Rk\chi_{w}\|_{L^2\rightarrow L^2}$. 

Since $\|\chi_{z,w}^{(1)}{\Rk\chi_{w}u}\|$ and $\|\chi_{z,w}^{(1)}\Dk \Rk \chi_{w}u\|$ are both real numbers, it is straightforward to see that for all  $\zeta\in U\setminus\R$ and $u\in L^2$,  
\begin{equation}\label{eq:0}
  \|\chi_{z,w}^{(1)}{\Rk\chi_{w}u}\|^2_{\omega_0}\leq\frac{1}{|\operatorname{Im }\zeta|}\Big|\zeta\cdot\|\chi_{z,w}^{(1)}\Rk\chi_{w}u\|^2_{\omega_0}-\|\chi_{z,w}^{(1)}\Dk \Rk \chi_{w}u\|^2_{\omega_0}\Big|.  
\end{equation}
 The last term can be computed by integration by parts and \eqref{eq:LB} as follows:
\begin{multline}\label{eq:1}
 \|\chi_{z,w}^{(1)}\Dk \Rk \chi_{w}u\|^2_{\omega_0}=\left(\Dk^*(\chi_{z,w}^{(1)})^2\Dk \Rk \chi_{w}u\Big|\Rk \chi_{w}u\right)_{\omega_0}\\
 =\left(\chi_{z,w}^{(1)}\Hmodk \Rk \chi_{w}u\Big|\chi_{z,w}^{(1)}\Rk \chi_{w}u\right)_{\omega_0}+\left(V_{({\chi_{z,w}^{(1)}})^2}\Dk \Rk \chi_{w}u\Big|\Rk \chi_{w}u\right)_{\omega_0},
\end{multline}where the second term of the second line is dominated as: \begin{equation*}
    \begin{split}
    \Big|\left((V_{({\chi_{z,w}^{(1)}})^2}\Dk \Rk \chi_{w}u \Big|\Rk \chi_{w}u\right)\Big|
    &\leq \|(V_{({\chi_{z,w}^{(1)}})^2}\Dk\Rk\chi_wu\|\|\chi^{(2)}_{z,w}\Rk\chi_wu\|\\ 
    &\leq \frac{C(\chi_1)}{|z-w|}\|\chi^{(2)}_{z,w}\Dk\Rk\chi_wu\|\cdot\|\chi^{(2)}_{z,w}\Rk\chi_wu\|.     
    \end{split}
\end{equation*}
Here, we use the Cauchy-Schwarz inequality and the fact that \[|V_{({\chi_{z,w}^{(1)}})^2}\eta|_{\omega_0}=|[\Lambda_{\omega_0},-\ii\partial (\chi_{z,w}^{(1)})^2\wedge\cdot]\,\eta|_{\omega_0}\leq C(\chi_1) |z-w|^{-1}|\eta|_{\omega_0},\] for some $C(\chi_1)>0$ and $\eta\in \Lambda^{0,1}T^{*}\Cs^n$. By combining this with  \eqref{eq:0},  \eqref{eq:1} , we have 
\begin{multline}\label{eq:1.5}
    \|\chi^{(1)}_{z,w}\Rk\chi_wu\|^2\leq \frac{1}{|\operatorname{Im}\zeta|}\Big|\left(\chi^{(1)}_{z,w}(\zeta-\Hmodk)\Rk\chi_wu|\chi^{(1)}_{z,w}\Rk\chi_wu\right)\Big|\\+\frac{C(\chi_1)}{|\operatorname{Im}\zeta|\cdot|z-w|}(\|\chi^{(2)}_{z,w}\Dk\Rk\chi_wu\|^2+\|\chi^{(2)}_{z,w}\Rk\chi_wu\|^2).
\end{multline}Since $\supp\chi^{(1)}_{z,w}\cap\supp\chi_w=\varnothing$ and $(\zeta-\Hmodk)\Rk=\Id$, the first term on the right-hand side above vanishes. 

Inductively, for every $j\in\N$, the same argument gives the following estimate.  
\begin{equation}\label{eq:2}
 \|\chi_{z,w}^{(j)}\Rk\chi_{w}u\|^2\leq\dfrac{C(\chi_j)}{|\operatorname{Im }\zeta|\cdot|z-w|}\left(\|\chi_{z,w}^{(j+1)}\Rk\chi_{w}u\|^2_{\omega_0}+\|\chi_{z,w}^{(j+1)}\Dk\Rk\chi_{w}u\|^2_{\omega_0}\right).   
\end{equation}Moreover, observe that the terms on the right-hand side above are bounded in the sense that $\|\chi_{z,w}^{(j+1)}\Rk\chi_{w}u\|_{\omega_0}\leq |\operatorname{Im}\zeta|^{-1}\|u\|_{\omega_0}$ and \begin{multline*}
  \|\chi_{z,w}^{(j+1)}\Dk\Rk\chi_{w}u\|^2_{\omega_0}\leq \|\Dk\Rk\chi_{w}u\|^2_{\omega_0}\\ = (\underbrace{\Hmodk\Rk}_{=\zeta\Rk-\Id}\chi_wu|\Rk\chi_wu)_{\omega_0}\leq \|(\zeta\Rk-\Id)\chi_wu\|_{\omega_0}^2+\|\Rk\chi_wu\|^2_{\omega_0}\leq \frac{C(U)}{|\operatorname{Im}\zeta|^2}\|u\|^2_{\omega_0},  
\end{multline*}for some $C(U)>0$. Hence, by iterating the estimates \eqref{eq:2}, the proof of the claim reduces to showing the following: 

For each $j\in\N$,  
\begin{equation}\label{eq:3}
    \|\chi^{(j)}_{z,w}\Dk\Rk\chi_w u\|^2_{\omega_0}\leq \dfrac{C(\chi_j,U)}{|\operatorname{Im}\zeta|\cdot|z-w|}\left(
    \|\chi^{(j+1)}_{z,w}\Rk\chi_wu\|^2+\|\chi^{(j+1)}_{z,w}\Dk\Rk\chi_wu\|^2\right),
\end{equation}for some $C(\chi_j,U)>0$. To show the estimate above, we apply an argument similar to the one in \eqref{eq:0} by noting that \begin{equation}\label{eq:4}
   \|\chi^{(j)}_{z,w}\Dk\Rk\chi_w u\|^2_{\omega_0}\leq \frac{1}{|\operatorname{Im}\zeta|}\Big|\zeta\cdot\|\chi^{(j)}_{z,w}\Dk\Rk\chi_w u\|^2_{\omega_0}-\|\chi^{(j)}_{z,w}\underbrace{\Dk^*\Dk}_{=\Hmodk}\Rk\chi_w u\|^2\Big|.
\end{equation}As in \eqref{eq:1}, we compute the second term by 
\begin{multline}\label{eq:5}
\|\chi^{(j)}_{z,w}\Hmodk\Rk\chi_wu\|^2=\left(\chi^{(j)}_{z,w}\Dk\Hmodk\Rk\chi_w u\Big|\chi^{(j)}_{z,w}\Dk\Rk\chi_w u\right)\\+\left((\db(\chi^{(j)}_{z,w})^2)\wedge \Hmodk\Rk\chi_w u\Big|\Dk\Rk\chi_w u\right).
\end{multline}Again, by Cauchy-Schwarz inequality and the fact $|\db\chi^{(j)}_{z,w}|\leq C(\chi_j)|z-w|^{-1}$, we have 
\begin{multline*}
 \Big|\left((\db(\chi^{(j)}_{z,w})^2)\wedge \Hmodk\Rk\chi_w u\Big|\Dk\Rk\chi_w u\right)\Big|\\ \leq \frac{C(\chi_j)}{|z-w|}\|\chi^{(j+1)}_{z,w}\Hmodk\Rk\chi_wu\|\|\chi^{(j+1)}_{z,w}\Dk\Rk\chi_wu\|.   
\end{multline*}
As in \eqref{eq:1.5}, we combine this with \eqref{eq:4} and \eqref{eq:5} to get
\begin{multline}
    \|\chi^{(j)}_{z,w}\Dk\Rk\chi_wu\|^2\leq \frac{1}{|\operatorname{Im}\zeta|}\Big|\left(\chi^{(j)}_{z,w}\Dk(\zeta-\Hmodk)\Rk\chi_wu|\chi^{(j)}_{z,w}\Dk\Rk\chi_wu\right)\Big|\\+\frac{C(\chi_j)}{|\operatorname{Im}\zeta|\cdot|z-w|}(\|\chi^{(j+1)}_{z,w}\Hmodk\Rk\chi_wu\|^2+\|\chi^{(j+1)}_{z,w}\Dk\Rk\chi_wu\|^2).
\end{multline}Since $\chi_{z,w}^{(j)}\Dk(\zeta-\Hmodk)
\Rk\chi_w=0$, the first term on the right-hand side above vanishes. Finally, note that $\chi^{(j+1)}_{z,w}\Hmodk\Rk\chi_w=\chi^{(j+1)}_{z,w}\zeta\Rk\chi_w$, which implies \[
\|\chi^{(j+1)}_{z,w}\Hmodk\Rk\chi_wu\|\leq C(U) \|\chi^{(j+1)}_{z,w}\Rk\chi_wu\|,
\] for some $C(U)>0$. We obtain \eqref{eq:3} and hence complete the proof of the claim.
\end{proof}
\begin{remark}\label{Rmk1}
As seen in the proof, the factor $|\operatorname{Im} \zeta|^{-M}|z-w|^{-N}$ in Theorem \ref{Thm1} arises from the $L^2$-estimate of the operator norm of $\chi_z(\zeta-\Hmodk)^{-1}\chi_w$, which is independent of the weight $\phimodk$. On the other hand, the factor $(1+|z|+|w|+k)^{\ell(r,U)}$ is determined by the coefficient growth of the differential operator $\Hmodk$. More precisely, the exponent $\ell(r,U)$ can be controlled in terms of finitely many seminorms of the coefficients.
\end{remark}
We denote the Bergman projection by \begin{equation}
\Bmodk:=B_{\phimodk,\omega_0}:L^2_{\phimodk,\omega_0}(\Cs^n)\to \Ker\Boxmodk,
\end{equation}whose Schwartz kernel is $\Bmodk(z,w):=B_{\phimodk,\omega_0}(z,w)$ (see \eqref{Bk2}). Under the unitary relation $\eqref{eq:uni_model}$, we can transform the Bergman projection and its kernel by \begin{equation}\label{123}
    \begin{split}
        \Bmodk  \, &\longleftrightarrow \, \sBmodk\,:=\ e^{-\phimodk}\Bmodk e^{\phimodk}; \\
        \Bmodk(z,w)\, &\longleftrightarrow \, \sBmodk(z,w)\,:=\ e^{-\phimodk(z)}\Bmodk(z,w) e^{\phimodk(w)}.
    \end{split}
\end{equation}Then, $\sBmodk:L^2_{\omega_0}\to \Ker\Hmodk$ is the orthogonal projection. In the \emph{quasihomogeneous case}, we denote $\Bmod=\Bmodk$ and $\sBmod=\sBmodk$. Here, we can naturally identify the  Bergman kernels as \[\mathcal{B}^\Lambda_k(\cdot,\cdot)\in\Cinf (\Cs^n\times \Cs^n).\] 
By repeating the arguments in \eqref{e1}--\eqref{e2}, replacing \(\Rk\) by \(\sBmodk\) and \(\zeta-\Hmodk\) by \(\Hmodk\), we obtain, for every \(r\in\N\), \[
|\sBmodk(z,w)|_{\Cr}\leq C(1+|z|+|w|+k)^{\ell(r)}\|\Tilde{\chi}_z \sBmodk \Tilde{\chi}_w\|_{L^2\to L^2},
\]for some $C>0$ and $\ell(r)>0$. Since the operator norms of $\sBmodk$ are bounded by one, we get:
\begin{equation}\label{BB}
|\sBmodk(z,w)|_{\Cr}\leq C(1+|z|+|w|+k)^{\ell(r)}.
\end{equation}

Suppose now that the \emph{model gap condition}
(see Definition~\ref{def:model-gap}) holds, so that we have  $\operatorname{Spec}(\Hk)\subset\{0\}\cup[c,\infty)$, for some $c>0$. We choose $f\in\Cinf_c(\R)$ such that $f(0)=1$ and $\supp f\subset [-\infty,\frac{c}{2}]$. Then, \begin{equation}\label{HS_f}
  f(\Hk)=\mathbf 1_{\{0\}}(\Hk)=\sBmodk.  
\end{equation}
Hence, the Helffer--Sj\"ostrand formula (see \eqref{HS}) gives \begin{equation}\label{HS_use}
     \sBmodk(z,w)=\frac{1}{2\pi\ii} \int_{\Cs} (\zeta-\Hmodk)^{-1}(z,w)\,
    d\zeta \wedge \db \widetilde{f}.
\end{equation}This shows that the integral of $(\zeta-\Hmodk)^{-1}(z,w)$ is the Bergman kernel in the sense of distributions. Combining  Theorem~\ref{Thm1}, \eqref{BB} and \eqref{HS_use}, we obtain
the following result.
\begin{theorem}[Off-diagonal estimate for the model Bergman kernel]\label{Thm_off_B_mod}
Assume the model gap condition holds. For all $r,N\in\N$, there exist $\ell=\ell(r)\in\N$ and
$C=C(r,N)>0$ such that
\begin{equation*}
\big|\sBmodk(z,w)\big|_{\mathscr{C}^r}\leq C\,\frac{(1+|z|+|w|+k)^{\ell}}{1+|z-w|^N}
\end{equation*}
for all $k\in\N$ and all $z,w\in\Cs^n$. If the quasihomogeneous model gap
condition holds, then 
\begin{equation*}
\big|\sBmod(z,w)\big|_{\mathscr{C}^r}\leq C\,\frac{(1+|z|+|w|)^{\ell}}{1+|z-w|^N}.
\end{equation*}
\end{theorem}

\section{Scaling of metrics and extended Laplacian}\label{sec_4}
In this section, we introduce the rescaling of the metrics and some estimates. We start from the ordinary Taylor expansion on $\Cs^n$, and regroup the terms according to a weight vector. 
\subsection*{Rescaling of the local weight}
For a weight vector $\Lambda=(\lambda_1,\cdots,\lambda_n)$ with $0<\lambda_1\le\cdots\le\lambda_n$, define the weighted radius on $\Cs^n$ by  $|z|_{\Lambda}:=\max_{j=1,\cdots,n}|z_j|^{1/{\lambda_j}}$. Also, for $\alpha,\beta\in\N_0^n$, recall that $|(\alpha,\beta)|_\Lambda=\sum\lambda_j(\alpha_j+\beta_j)$. After a linear change of coordinates, we may assume $\{z;|z|_\Lambda<1\}\subset D$. For any $Q>1$, \[
 \varphi(z)=\sum_{|(\alpha,\beta)|_{\Lambda}<Q}a_{\alpha,\beta}z^\alpha \bar z^\beta+R_Q(z), \quad |z|_\Lambda<1, 
 \]where $R_Q$ satisfies the remainder estimate: For all  $\gamma,\eta\in\N_0^n$, there exists $C_{Q,\gamma,\eta}>0$ such that \[\Big|\partial^\gamma_z\partial^\eta_{\bar z}R_Q(z)\Big|\leq C_{Q,\gamma,\eta}|z|_\Lambda^{Q-|(\gamma,\eta)|_{\Lambda}},\quad\text{for } 0<|z|_\Lambda<1.
 \]
We now rescale the local weight function $k\varphi$.
Recall the anisotropic scaling map
$S_k^\Lambda(z)=(k^{-\lambda_1}z_1,\ldots,k^{-\lambda_n}z_n)$, $k\ge1 $. For $Q>1$, let
\[
R_{Q,k}(z):=kR_Q(S_k^\Lambda(z)).
\]
By the chain rule and the weighted Taylor remainder estimate,
\[
\left|
\partial_z^\gamma\partial_{\bar z}^\eta R_{Q,k}(z)
\right|
\le
C_{Q,\gamma,\eta}k^{1-Q}
|z|_\Lambda^{Q-|(\gamma,\eta)|_\Lambda},
\]
for $0<|z|_\Lambda<k$. Fix $\epsilon_0>0$ satisfying
\begin{equation}\label{epsilon_0}
0<\epsilon_0<\frac{d_1^2}{1+d_1},
\end{equation}
and set $\theta:=\epsilon_0/\lambda_1$.
Since $d_1\le\lambda_1$, we have
$\theta<d_1/(1+d_1)<1$ and
$|z|_\Lambda\le k^\theta$ whenever
$|z|\le k^{\epsilon_0}$. Consequently, for every $T>0$ and $r\in\mathbb N_0$,
\begin{equation}\label{weighted_taylor}
\Big\|
k\varphi(S_k^\Lambda(\cdot))
-
\sum_{|(\alpha,\beta)|_\Lambda<T}
k^{1-|(\alpha,\beta)|_\Lambda}
a_{\alpha,\beta}z^\alpha\bar z^\beta
\Big\|_{\Cr(B(k^{\epsilon_0}))}
\le
C_{T,r}k^{1-(1-\theta)T}.
\end{equation}
Indeed, this follows by taking the Taylor expansion
to a sufficiently high weighted order
$Q>\max\{T,r\lambda_n\}$ and estimating
the omitted monomials and the remainder. Taking
$T=1+d_1$ in \eqref{weighted_taylor}, we obtain
\begin{equation}\label{re}
\left\|
k\varphi(S_k^\Lambda(\cdot))
-\varphi_k^\Lambda
\right\|_{\Cr(B(k^{\epsilon_0}))}
\le C_r k^{-c_0},
\end{equation}
where $c_0=(1-\theta)(1+d_1)-1>0$. More generally, let $0<d_1<d_2<\cdots$
be the orders defined in \eqref{order}, and let
\[
\varphi_j^\Lambda(z)
:=
\sum_{|(\alpha,\beta)|_\Lambda=1+d_j}
a_{\alpha,\beta}z^\alpha\bar z^\beta,
\]where $\varphi_j^\Lambda$ is an odd polynomial if $d_j\notin \mathcal{O}^{\even}_\Lambda$.
Taking $T=1+d_{N+1}$ in
\eqref{weighted_taylor}, we obtain
\begin{equation}\label{-2}
\left\|
k\varphi(S_k^\Lambda(\cdot))
-\varphi_k^\Lambda
-\sum_{j=1}^N k^{-d_j}\varphi_j^\Lambda
\right\|_{\Cr(B(k^{\epsilon_0}))}
\le C_{N,r}k^{-c_N},
\end{equation}
where $c_N=(1-\theta)(1+d_{N+1})-1
\longrightarrow+\infty$.

\subsection*{Scaling of the volume factors}
In the model coordinate $D$, we write the Hermitian $(1,1)$-form $\omega$ locally as $
\omega(z)=\ii\sum_{i,j=1}^n\omega_{i,j}(z)dz_i\wedge d\bar z_j$. Recall that $\omega_e=\sum_{j=1}^n\frac{\ii}{2}dz_j\wedge d\bar z_j$ is the standard Hermitian $(1,1)$-form on Euclidean space and $dV$ is the standard Lebesgue measure. Denote \begin{equation}
    v(z):=\frac{dV_{\omega}}{dV}(z)=\frac{\omega^n}{\omega^n_e}(z),
\end{equation}which is the volume function with respect to the coordinates. Note that $\omega_0=v^{\frac{1}{n}}(0)\cdot \omega_e$.
 Similar to the Taylor expansion of $\vphi$, for every $N,r\in\mathbb N_0$,
\begin{equation}\label{-1}
\left\|
v(S_k^\Lambda(\cdot))
-v(0)
-\sum_{j=1}^N k^{-d_j}p_j^\Lambda
\right\|_{\Cr(B(k^{\epsilon_0}))}
\le C_{N,r}k^{-c_N},
\end{equation}
where
$p_j^\Lambda$ is the
$\Lambda$-homogeneous term of weighted
degree $d_j$. Note that $p_j^\Lambda$ is an odd polynomial if $d_j\notin \mathcal{O}^{\even}_\Lambda$. 

\subsection*{Extended metrics and Extended operators}
Let $\chi: \mathbb{C}^n \to [0,1]$ be a smooth cut-off function such that $\operatorname{supp}(\chi) \subset B(1)$ and $\chi \equiv 1$ on $B(\frac{1}{2})$. Define the extended weight functions and positive Hermitian $(1,1)$-forms on $\Cs^n$ by  
\begin{equation}\label{eq:sca_metric}
\begin{split}
\widetilde{\varphi}^\Lambda_{k}&:=\chi\left(\frac{z}{k^{\epsilon_0}}\right)\cdot\left(k\varphi(S^\Lambda_k(z))\right)+(1-\chi\left(\frac{z}{k^{\epsilon_0}}\right))\cdot\varphi^\Lambda_k(z)\quad;\\
\widetilde{\omega}^\Lambda_{k}&:=\Big[\chi\left(\frac{z}{k^{\epsilon_0}}\right)\cdot v^{\frac{1}{n}}(S^\Lambda_k(z))+(1-\chi\left(\frac{z}{k^{\epsilon_0}}\right))\cdot v^{\frac{1}{n}}(0)\Big]\omega_e(z).
\end{split}
\end{equation}
We now consider the asymptotic behavior of the Hilbert spaces $L^2_{\tphik,\tomegak}(\Cs^n)$. For any $k\in\N$, $L^2_{\phimodk,\omega_0}(\Cs^n)=L^2_{\tphik,\tomegak}(\Cs^n)$ as sets of functions since the weights and Hermitian forms are identical outside a compact set. By the uniform convergence of the metrics, the norms $\|\cdot\|_{\phimodk,\omega_0}$ and $\|\cdot\|_{\tphik,\tomegak}$ are uniformly compatible in the sense that there exists $C>0$ such that $C^{-1}\|\cdot\|_{\tphik,\tomegak}\leq \|\cdot\|_{\phimodk,\omega_0}\leq C \|\cdot\|_{\tphik,\tomegak}$ for all $k\in\N$.  

In the rest of this section, we consider the Kodaira Laplacian with respect to the weight function $\tphik$ and Hermitian $(1,1)$-form $\tomegak$. Denote \begin{equation}\label{tBoxk}
\tBoxk:=\Box_{\tphik,\tomegak}=\db^*_{\tphik,\tomegak}\db:\Dom\Box_{\tphik,\tomegak}\to L^2_{\tphik,\tomegak}(\Cs^n).  
\end{equation}
For $\zeta\in\Cs\setminus \R$, the resolvent is denoted by \begin{equation}\label{tRk}
    \tRk:=(\zeta-\tBoxk)^{-1}:L^2_{\tphik,\tomegak}(\Cs^n)\to \Dom\tBoxk.
\end{equation}
Define $\beta_k(z):=\frac{dV_{\omega_0}}{dV_{\tomegak}}(z)$. Similar to \eqref{eq:uni_model}, we consider the unitary identification \begin{equation}\label{unitary_t}
   \begin{split}
L^2_{\tphik,\tomegak}(\Cs^n) &\longleftrightarrow L^2_{\omega_0}(\Cs^n) \\
f &\longmapsto e^{-\tphik}\beta_k^{-1/2}f .
\end{split} 
\end{equation} This transforms the differential operator, and the Laplacian as follows:
\begin{align*}
   \db\,&\longleftrightarrow\,\tDk:=e^{-\tphik}\beta_k^{-1/2}\db e^{\tphik}\beta_k^{1/2} \\
   \tBoxk\,&\longleftrightarrow\, \tHk:=e^{-\tphik}\beta_k^{-1/2}\tBoxk e^{\tphik}\beta_k^{1/2}
   \end{align*}
Also, for $\zeta\in\Cs\setminus\R$, we denote by \begin{equation}
    \tsRk:=(\zeta-\tHk)^{-1}:L^2_{\omega_0}\to \Dom\tHk.
\end{equation}The Schwartz kernel is $\tsRk(z,w)$, which is equivalent to $e^{-\tphik}\beta_k^{-1/2}\tRk e^{\tphik}\beta_k^{1/2}$. Under the unitary identification \eqref{unitary_t},
the conjugated operator $\tHk$ involves the rescaled
weight $\tphik$ and Hermitian metric $\tomegak$.
 The coefficients of $\tHk$, together with their derivatives,
satisfy polynomial growth estimates in $|z|+k$.
In the quasihomogeneous case, these estimates are uniform in $k$, with polynomial growth in $|z|$.
Thus, the elliptic and commutator estimates in the proof
of Theorem~\ref{Thm1} carry over to $\tHk$.
By Remark~\ref{Rmk1}, we obtain the following
off-diagonal estimate for $\tsRk(z,w)$.

\begin{theorem}[Off-diagonal estimate]\label{Thm2}For any $r,N\in\N$ and a bounded domain $U\Subset\Cs$, there exists $\ell(r,U),M(N)\in\N$ and $C({r,N},U)>0$ such that the Schwartz kernel $\tRk(\cdot,\cdot)$ of $\tRk$ satisfies the off-diagonal estimate \begin{equation*}
    \Big| \tsRk(z,w)\Big|_{\mathscr{C}^r}\leq \dfrac{C(r,N)}{|\operatorname{Im }\zeta|^{M(N)}}\cdot \dfrac{(1+|z|+|w|+k)^{\ell(r,U)}}{|z-w|^N},
\end{equation*}for all $|z-w|\geq 1$, $k\in\N$ and $\zeta\in U\setminus\R$. Moreover, if the model polynomial $\phimodk$ is quasihomogeneous, which means $\phimodk=\phimod$, then  the inequality above reduces to the following:  \begin{equation*}
    \Big| \tsRk(z,w)\Big|_{\mathscr{C}^r}\leq \dfrac{C(r,N)}{|\operatorname{Im }\zeta|^{M(N)}}\cdot \dfrac{(1+|z|+|w|)^{\ell(r,U)}}{|z-w|^N}.
\end{equation*}    
\end{theorem}
We denote the Bergman projection by \begin{equation}\label{tBk}
    \tBk:= B_{\tphik,\tomegak}:L^2_{\tphik,\tomegak}\to \Ker\tBoxk,
\end{equation}whose Schwartz kernel is $\tBk(z,w)$ in the sense that $\tBk u(z)=\int \tBk(z,w)u(w)dV_{\tomegak}$ in distribution sense. By unitary relation \eqref{unitary_t}, we have \begin{equation*}
  \tBk  \, \longleftrightarrow \, \tsBk\,=\ e^{-\tphik}\beta^{-1/2}_k\tBk e^{\tphik}\beta^{1/2}_k,
      \end{equation*}where the Schwartz kernel of $\tsBk$ with respect to the measure $dV_{\omega_0}$ is given by \begin{equation*}
       \tsBk(z,w)=e^{-\tphik(z)}\beta_k (z)^{-1/2}\tBk(z,w) e^{\tphik(w)}\beta_k(w)^{-\frac{1}{2}}.   
      \end{equation*}
Note that $\tsBk:L^2_{\omega_0}\to \Ker\tHk$ is the orthogonal projection. The following proposition shows that the spectral gap of $\Boxmodk$ leads to the spectral gap of $\tBoxk$ by some elementary perturbation arguments. 
 
\begin{proposition}\label{prop:stable_gap}
    Assume  $\operatorname{Spec}{\Boxmodk}\cap (0,c)=\varnothing$ for some $c>0$, and sufficiently large $k$.
    Then, there exists $c'>0$ and $k_0\in\N$ such that for all $k\geq k_0$, \[
\operatorname{Spec}\tBoxk\cap (0,c')=\varnothing.
    \] 
\end{proposition}
\begin{proof}
 Let $u_k\in \Dom\db_k$  and $u_k\perp \operatorname{Ker}\db$ with respect to the metric $(\cdot|\cdot)_{\tphik,\tomegak}$. For every  \mbox{$v\in\ker\db\cap L^2_{\tphik,\tomegak}$}, a direct computation yields \[
 (u_k|v)_{\phimodk,\omega_0}=(u_k|v)_{\phimodk,\omega_0}-\underbrace{(u_k|v)_{\tphik,\tomegak}}_{=0}=\int_{\Cs^n}u_k\Bar{v}(1-\underbrace{e^{-2(\tphik-\phimodk)}\dfrac{(\tomegak)^n}{(\omega_0)^n}}_{=:V_k(z)})e^{-2\phimodk} dV_{\omega_0}
 \]By the convergence of metrics \eqref{-2} and \eqref{-1}, we have $\sup |(1-V_k)|=O(k^{-c})$ for some $c>0$, and hence $(u_k|v)_{\phimodk,\omega_0}=O(k^{-c})\|u_k\|\|v\|$. Then, \[
 \|\Bmodk u_k\|^2_{\phimodk,\omega_0}=(u_k|\Bmodk u_k)_{\phimodk,\omega_0}=O(k^{-c}) \|u_k\|^2_{\phimodk,\omega_0},
 \]which means $\|u_k\|^2_{\tphik,\tomegak}\lesssim \|u_k-\Bmodk u_k\|^2_{\phimodk,\omega_0}$. Applying the assumption and the uniform compatibility of $\|\cdot\|_{\phimodk,\omega_0}$ and $\|\cdot\|_{\tphik,\tomegak}$, we derive that for large enough $k\in\N$, 
 \begin{equation*}
\|u_k\|^2_{\tphik,\tomegak}\lesssim \|u_k-\Bmodk u_k\|^2_{\phimodk,\omega_0}\leq c^{-1} \|\db u_k\|_{\phimodk,\omega_0}^2\lesssim c^{-1}{\|\db u_k\|^2_{\tphik,\tomegak}}, 
\end{equation*}which completes the proof by the fact that $\|\db u_k\|^2_{\tphik,\tomegak}=(\tBoxk u_k|u_k)_{\tphik,\tomegak}$. 
 \end{proof}
By Proposition~\ref{prop:stable_gap}, if the model gap assumption holds,   we may choose a suitable $f\in\Cinf_c(\R)$ with $f(0)=1$ and $\supp f\subset [-\infty,\frac{c'}{2}]$ such that $f(\tHk)=\mathbf 1_{\{0\}}(\tHk)=\tsBk$. The Helffer--Sj\"ostrand formula gives  \[
\tsBk
    = \frac{1}{2\pi \ii}\int_{\Cs}
    (\zeta-\tHk)^{-1}\,
    d\zeta \wedge \db \Tilde{f},
\]for all large enough $k$. Hence, we have the  following analogous result to Theorem \ref{Thm_off_B_mod}:
\begin{theorem}[Off-diagonal estimate for scaled Bergman kernel]\label{Thm_off_B}Assume there exists $c>0$ such that $\Spec\Boxmodk\cap(0,c)=\varnothing$ for large enough $k\in\N$. For any $r,N\in\N$, there exists $\ell(r)>0$ and $C=C(r,N)>0$ such that 
\begin{equation*}\Big| \tsBk(z,w)\Big|_{\mathscr{C}^r}\leq C(r,N)\dfrac{(1+|z|+|w|+k)^{\ell(r)}}{1+|z-w|^N},
\end{equation*}for all large enough $k\in\N$. Moreover, if the quasihomogeneous model holds,\begin{equation*}\Big| \tsBk(z,w)\Big|_{\mathscr{C}^r}\leq C(r,N)\dfrac{(1+|z|+|w|)^{\ell(r)}}{1+|z-w|^N}.
\end{equation*} 
\end{theorem}
\section{Adjustment of the base Hermitian \texorpdfstring{$(1,1)$}{(1,1)}-form and scaling identity}\label{sec_adj}

Before embarking on the localization of the resolvents, we make some adjustment of the base Hermitian $(1,1)$-form $\omega$, which preserves the volume measure $dV_{\omega}$. 

Recall that $S^\Lambda_k(z_1,\cdots,z_n)$ is the scaling map considered in the previous sections. We can consider the pull-back of the Hermitian metric $h$ on $L$ and the Hermitian $(1,1)$-form $\omega$ by $S^\Lambda_k$ near the given point $p\in X$. For any section $u\in\Cinf(X,L^k)$, we can locally pull-back the section near $p$ and denote 
\[
S^\Lambda_k\,u:=(S^\Lambda_k)^*u=u(S^\Lambda_k(z)),
\]for all $x\in D$.
The corresponding weight function of the pull-back metric $(S^{\Lambda}_k)^*h^k$ in $D$ is given by $k\varphi(S^\Lambda_k(z))$. By the construction in the previous section (see \eqref{eq:sca_metric}),\[k\varphi(S^\Lambda_k(z))=\tphik(z),\quad\text{for }|z|<\frac{1}{2}k^{{\epsilon_0}}.\]
If $\Lambda=(\lambda,\cdots,\lambda)$ with $\lambda^{-1}\in2\N$, and $\omega=v^{\frac1n}\omega_e$ near $p$, one has $(S^\Lambda_k)^*\omega(z)=k^{-2\lambda}\tomegak(z)$ for all $|z|<\frac{1}{2}k^{{\epsilon_0}}$. Hence, by locally identifying the pull-back section $S^\Lambda_ks^k:S^{\Lambda}_k(D)\To (S^\Lambda_k)^*L^k$ to the unit section $\mathbf{1}{:\Cs^n}\to\{1\}$, one has the scaling identity:
 \[\Box_{(S^\Lambda_k)^*h^k,(S^\Lambda_k)^*\omega}=k^{2\lambda}\cdot \tBoxk,\quad \text{for }\,|z|\leq \frac{k^{{\epsilon_0}}}{2}.\]
However, in our general case, we do not have such an identity if $\lambda_i\neq \lambda_j$ for some $i\neq j$, or if $\omega$ is not conformal to $\omega_e$ near $p$. To overcome this difficulty, we modify the Hermitian form $\omega$ near $p$ while preserving the volume form $dV_\omega$.

 Let $F\in\Cinf_c(D,[0,1])$ be a bump function such that $F=1$ near the given point $p$. Consider the positive $(1,1)$-form
\[
\eta_k:=F\cdot\sum_{j=1}^n\frac{\ii}{2}\,k^{2\lambda_j}dz_j\wedge d\bar z_j+(1-F)\cdot\omega\quad\text{in }D,
\qquad \eta_k:=\omega\quad\text{in }X\setminus D,
\]
and define the modification $\omega_k$ of $\omega$ by
\begin{equation*}
    \omega_k:=\Big(\frac{dV_\omega}{dV_{\eta_k}}\Big)^{\frac1n}\eta_k .
\end{equation*}
Then $\omega_k$ is a smooth Hermitian $(1,1)$-form on $X$ with $\omega_k\equiv\omega$ in $X\setminus D$. Near $p$, where $F=1$, we have $dV_{\eta_k}=k^{2|\Lambda|}dV$, and hence
\begin{equation*}
    \omega_k(z)=k^{-\frac{2}{n}|\Lambda|} \cdot v(z)^{\frac1n}\sum_{j=1}^n\frac{\ii}{2}\cdot k^{2\lambda_j}dz_j\wedge d\bar z_j .
\end{equation*}
By construction, for every $k\in\N$, one has
$dV_{\omega_k} =dV_\omega$. Consequently, 
$(\,\cdot\,|\,\cdot\,)_{h^k,\omega_k}=(\,\cdot\,|\,\cdot\,)_{h^k,\omega}$ on $\Cinf_c(X,L^k)$. Thus, we identify isometrically
$L^2_{h^k,\omega_k}(X,L^k)=L^2_{h^k,\omega}(X,L^k)$. Although the corresponding Kodaira Laplacians need not coincide, their
kernels on $L^k$-valued functions are the same:
$\ker \Box_{h^k,\omega}=\ker \Box_{h^k,\omega_k}=\ker \bar\partial\cap L^2_{h^k,\omega}(X,L^k)$. In particular, the Bergman projection is not changed. For $k$ large, $S^\Lambda_k(\{|z|<k^{\epsilon_0}\})\subset\{F=1\}$, and since $(S^\Lambda_k)^*dz_j=k^{-\lambda_j}dz_j$, we get by \eqref{eq:sca_metric}
\[
(S^\Lambda_k)^*\omega_k=k^{-\frac{2}{n}|\Lambda|}\,v\big(S^\Lambda_k(z)\big)^{\frac1n}\omega_e=k^{-\frac{2}{n}|\Lambda|}\,\tomegak(z),\quad\text{for }|z|<\frac12k^{\epsilon_0}.
\]
Since $\Box_{\psi,c\,\omega}=c^{-1}\Box_{\psi,\omega}$ on functions for every constant $c>0$, we have the following local scaling identity in $D$: \begin{equation}\label{eq:sca_kod}
\Box_{(S_k^\Lambda)^*h^k,(S_k^\Lambda)^*\omega_k}=k^{\frac{2}{n}|\Lambda|}\cdot (\Box_{\tphik,\tomegak}),\quad \text{for }\,|z|\leq \frac{k^{{\epsilon_0}}}{2}.
\end{equation}

\section{Localization: Proof of Theorem~\ref{Thm_Loc_Off}}\label{sec_5}

In this section, we introduce an approach to localize the Bergman kernel via localization of resolvents and the Helffer-Sjöstrand formula. Throughout this section, we impose the assumptions in Theorem \ref{Thm_Loc_Off}. 

Recall that we can choose $f\in\Cinf_c(\R,[0,1])$ with $f(0)=1$ such that $f(\tBoxk)=\tBk$ for all large enough $k$ by the \emph{model gap assumption} (see Definition \ref{def:model-gap}) and Proposition~ \ref{prop:stable_gap}. Here, $\tBk$ is the weighted Bergman kernel with respect to the measure $e^{-2\tphik}\frac{(\tomegak)^n}{n!}$. (see  \eqref{tBk}). 
\subsection{Localization of $f(k^{-\frac{2}{n}|\Lambda|}\Box_{h^k,\omega_k})$}

Recall that $\chi\in\Cinf_c(\Cs^n,[0,1])$ is the cut-off function considered in \eqref{eq:sca_metric}. Define $D_k:=S^\Lambda_k(B(1))\subset D$ be the shrinking domain. Now, we aim to localize the functional calculus $f(k^{-\frac{2}{n}|\Lambda|}\Box_{h^k,\omega_k})$ near the point $p$.  To this end, we begin by comparing the resolvents $(\zeta-k^{-\frac{2}{n}|\Lambda|}\Box_{h^k,\omega_k})^{-1}$ with  $\tRk=(\zeta-\Box_{\tphik,\tomegak})^{-1}$ in $D_k$.  Define a set of cut-off functions:  \begin{equation}\label{cut_off}
\begin{split}
     \chi_k(z)&:= \chi( (S^\Lambda_k)^{-1}z);\\
     \hchik(z)&:=\chi(\frac{2z}{k^{\epsilon_0}});\\
     \tchik(z)&:=\chi(\frac{2}{k^{\epsilon_0}}\cdot (S^\Lambda_k)^{-1}z).
\end{split}
\end{equation}  Note that $\chi_k,\tchik\in\Cinf_c(X;[0,1])$,  $\supp \chi_k\subset D_k$ and $\chi_k\equiv 1$ on $\frac{1}{2}D_k$. Define 
\begin{equation}\label{eq:Rlock}
       \Rlock:=  \underbrace{\widetilde{\chi}_k\,(S^\Lambda_k)^{-1}}_{=(S^\Lambda_k)^{-1}\hchik}\cdot \tRk \,S^\Lambda_k  \chi_k\cdot :L^2_{h^k,\omega_k}(X,L^k)\To L^2_{h^k,\omega_k}(X,L^k), 
    \end{equation}where $\tRk:=(\zeta-\Box_{\tphik,\tomegak})^{-1}$ is defined in \eqref{tRk}. Note that the operators $\Rlock$ are well-defined and bounded. By elliptic regularity, we have 
    \[
    \Rlock \, u\in \Dom \Box_{h^k,\omega_k}\quad\text{for }\, u\in L^2_{h^k,\omega_k}(X,L^k). 
    \]
    For simplicity of notation, we denote \[
    \Rglobk:=(\zeta-k^{-\frac{2}{n}|\Lambda|}\Box_{h^k,\omega_k})^{-1}.
    \]Since the operator $\Box_{h^k,\omega_k}:\Dom \Box_{h^k,\omega_k}\To L^2_{h^k,\omega_k}(X,L^k)$ is self-adjoint, we have \[|\operatorname{Im}\zeta|\,\|u\|_{h^k,\omega_k}^2=\Big|
\operatorname{Im}
\left(
(\zeta-k^{-\frac{2}{n}|\Lambda|}\Box_{h^k,\omega_k})u,u
\right)_{h^k,\omega_k} \Big|,\]for $u\in\Dom\Box_{h^k,\omega_k}$ and $\zeta\in\Cs\setminus\R$. By the Cauchy--Schwarz inequality,
\begin{equation}\label{eq:gap_res}
\|u\|_{h^k,\omega_k}\leq \dfrac{1}{|\operatorname{Im }\zeta|}\|(\zeta-k^{-\frac{2}{n}|\Lambda|}\Box_{h^k,\omega_k})u\|_{h^k,\omega_k},
\end{equation} for all $\zeta\in\Cs\setminus\R$ and $u\in \Dom \Box_{h^k,\omega_k}$. Hence, for $u\in L^2_{h^k,\omega_k}(X,L^k)$, we have \begin{equation}\label{eq:67}
\|\Rglobk\chi_k u-\Rlock u\|_{h^k,\omega_k} \leq \dfrac{1}{|\operatorname{Im}\zeta|}\|\chi_k u-(\zeta-k^{-\frac{2}{n}|\Lambda|}\Box_{h^k,\omega_k})\Rlock u\|_{h^k,\omega_k}
\end{equation}
By the standard identity of the pull-back operator, we locally have  \begin{equation}\label{pull-back}
   \Box_{h^k,\omega_k}=(S^\Lambda_k)^{-1}\circ\Box_{(S^\Lambda_k)^*h^k,(S^\Lambda_k)^*\omega_k}\circ(S^\Lambda_k). 
\end{equation} Combining this with \eqref{eq:sca_kod}, we can compute the last term on the right-hand side of \eqref{eq:67}. We get
\begin{equation}\label{eq:68}
    \begin{split}
     (\zeta-k^{-\frac{2}{n}|\Lambda|}&\Box_{h^k,\omega_k})\Rlock = (S_k^\Lambda)^{-1}(\zeta-k^{-\frac{2}{n}|\Lambda|}\Box_{(S_k^\Lambda)^*h^k,(S_k^\Lambda)^*\omega_k})\widehat{\chi}_k\tRk\chi S^\Lambda_k u \\
     &=(S_k^\Lambda)^{-1}\widehat{\chi}_k\underbrace{(\zeta-\tBoxk)\tRk}_{=\Id}\chi S^\Lambda_k u+(S^\Lambda_k)^{-1}{L}_k\tRk\chi u(S^\Lambda_k(z))\\
     &= \chi_ku+(S^\Lambda_k)^{-1}{L}_k\tRk\chi u(S^\Lambda_k(z)),
    \end{split}
\end{equation}
 where ${L}_k$ is a first-order differential operator supported in $\{\frac{k^{{\epsilon_0}}}{4} < |z| < \frac{k^{{\epsilon_0}}}{2}\}$. The operator arises from applying the Leibniz rule in the region where $\nabla\widehat{\chi}_k \neq 0$, and its coefficients grow at most polynomially as $k \to \infty$. By the unitary identification \eqref{unitary_t}, we get

\begin{equation}\label{eq:69}
    \|(S^\Lambda_k)^{-1}{L}_k\tRk\chi u(S^\Lambda_k(z))\|_{h^k,\omega_k}\leq \|\mathcal{L}_k\tsRk\chi \|_{L_{\omega_0}^2\to  L^2_{\omega_0}}\cdot \|u\|_{h^k,\omega_k}, 
\end{equation}where $\mathcal{L}_k:=e^{-\tphik}\beta_k^{-\frac{1}{2}}L_k e^{\tphik}\beta_k^{\frac{1}{2}}$ and  $\|\cdot\|_{L_{\omega_0}^2\to  L^2_{\omega_0}}$ is the operator norm for the space of bounded operators from $L_{\omega_0}^2(\Cs^n)$ to itself. Since the coefficients of $\mathcal{L}_k$ grow at most polynomial rate and $d(\supp\mathcal{L}_k,\supp\chi)\geq ck^{{\epsilon_0}}$ for some $c>0$, we can apply the off-diagonal estimate for $\tsRk(z,w)$ (see Theorem~\ref{Thm2}) to obtain that for any $N\in\N$, there exist $\ell(N)>0$ and $C_N>0$ such that  \begin{equation}\label{eq:70}
   \|\mathcal{L}_k\tsRk\chi \|_{L_{\omega_0}^2\to  L^2_{\omega_0}}\leq C_N\dfrac{k^{-N}}{|\operatorname{Im}\zeta|^{\ell(N)}}. 
\end{equation}Combining \eqref{pull-back}-\eqref{eq:70}, we can estimate the terms in \eqref{eq:67} as follows. For any $N\in\N$, there exist $\ell(N)>0$ and $C_N>0$ such that \begin{equation}
   \|\Rglobk\chi_k u-\Rlock u\|_{h^k,\omega_k} \leq C_N\dfrac{k^{-N}}{|\operatorname{Im}\zeta|^{\ell(N)}}\|u\|_{h^k,\omega_k}, 
\end{equation}which is the $L^2$-localization property of the resolvents $\Rglobk$ in $D_k$. Next, we apply the Helffer--Sj\"ostrand formula to get the $L^2$-localization of the functional calculus in $D_k$ as the following:\begin{equation}\label{eq:localization_f}
    \begin{split}
        \Big\|f(k^{-\frac{2}{n}|\Lambda|}{\Box}_{h^k,\omega_k})\chi_k&-(S^\Lambda_k)^{-1}\widehat\chi_kf(\tBoxk) S_k^{\Lambda}\chi_k\Big\|_{L^2\to L^2}\\& =\Big\|\frac{1}{2\pi \ii}\int_{\Cs}\left(\Rglobk\chi_k-\Rlock\right)d\zeta \wedge \db \Tilde{f}\Big\|_{L^2 \to L^2}\\&\leq \frac{1}{2\pi }\int_{\Cs}\Big\|\left(\Rglobk\chi_k-\Rlock\right)\Big\|_{L^2 \to L^2 }d\zeta \wedge \db \Tilde{f}\\&\leq C_N k^{-N},
    \end{split}
\end{equation} for any $N\in\N$ and some $C_N>0$. Here, we use the property of almost analytic extension \eqref{eq:HS_aa} to dominate the term $|\operatorname{Im}\zeta|^{-\ell(N)}$ in the integral.

Throughout this section, for two families of bounded operators
$A_k$, $B_k$ on $L^2_{h^k,\omega}(X,L^k)$, we write
\[
A_k\equiv B_k
\]
if, for every $N\in\N$, there exist constants $C_N>0$ and
$k_N\in\N$ such that
\[
\|A_k-B_k\|_{L^2\to L^2}
\leq C_Nk^{-N},
\qquad \text{for }\, k\geq k_N.
\]
By \eqref{eq:localization_f} and the fact that $f(\tBoxk)=\tBk$ for all large enough $k\in\N$, we  observe the following :
\begin{equation}\label{eq:133}
 \widetilde\chi_k\,(S^\Lambda_k)^{-1}\,\tBk \,S_k^{\Lambda}\chi_k \,\equiv \,f(k^{-\frac{2}{n}|\Lambda|}{\Box}_{h^k,\omega_k})\chi_k,  
\end{equation}where the cut-off function $\widetilde{\chi}_k$ is defined in \eqref{cut_off}.

\subsection{Localization of Bergman kernels}
We first claim the following Lemma.
\begin{lemma}\label{lem:kL}
    For any $m\in\N$, we have \begin{equation*}
\left(\Box_k\right)^m\tchik\left(\Sk\right)^{-1} \tBk\Sk\chik\equiv 0.
\end{equation*}
\end{lemma}
\begin{proof}By the pull-back identity \eqref{pull-back}, we get
\begin{equation*}
\Big[\left(\Box_k\right)^m\tchik\left(\Sk\right)^{-1} \tBk\Sk\chik u\Big](z)=(\Sk)^{-1}(\Box_{k (\Sk)^*\varphi,(\Sk)^*\omega})^m\hchik \tBk\chi u(\Sk(z)). 
\end{equation*}Since $\db\tBk=0$ and $\hchik(z)=1$ for $|z|\leq\frac{k^{\epsilon_0}}{4}$, we have
\begin{equation*}
    \Box_{k (\Sk)^*\varphi,(\Sk)^*\omega}\hchik \tBk(z)=0,\quad\textit{for }\,|z|\leq \frac{k^{\epsilon_0}}{4}.
\end{equation*}By the  Leibniz rule, we get: 
\begin{equation*}
  (\Box_{k (\Sk)^*\varphi,(\Sk)^*\omega})^m\hchik \tBk=L_{m,k}\tBk,  
\end{equation*}
where $L_{m,k}$ is a $(2m-1)$-order differential operator supported in $\{\frac{k^{{\epsilon_0}}}{4}<|z|<\frac{k^{\epsilon_0}}{2}\}$. The coefficients
grow at most polynomially as $k\to\infty$. Combining the above arguments, we only need to estimate as follows:
\begin{equation*}
    \|(\Sk)^{-1}L_{m,k}\tBk\chi u(S^\Lambda_k(z))\|_{h^k,\omega}\leq \|\mathcal{L}_{m,k}\tsBk\chi\|_{L^2_{\omega_0}\to L^2_{\omega_0}}\cdot \|u\|_{h^k,\omega}.
\end{equation*} Here, $\mathcal{L}_{m,k}:= e^{-\tphik}\beta_k^{-\frac{1}{2}}L_{m,k}e^{\tphik}\beta_k^{\frac{1}{2}}$. We complete the proof by the off-diagonal decay  estimate of $\tsBk(z,w)$ when $|z-w|\geq k^{{\epsilon_0}}/4$ (see Theorem \ref{Thm_off_B}). 
\end{proof}
Before proving the localization, we recall some standard estimates. All Sobolev and
$\Cr$-norms of sections of $L^k$ below are taken in the fixed coordinates and frame on $D$, as in
\eqref{Cr-norm} and \eqref{Sobolevnorm}. By the Sobolev inequality with semiclassical parameter
$k$, for any $r\in\N$ and any compact set $K\subset D$, there exist $C=C(K,r)>0$ and
$\ell=\ell(r)>0$ such that
\begin{equation}\label{s}
\|u\|_{\Cr(K)}\leq Ck^{\ell}\|u\|_{H^{n+r+1}(K)},\qquad u\in\Cinf_c(K,L^k).    
\end{equation}
Moreover, by the interior elliptic estimate for $\Box_k^m$, for any $m\in\N$ there exist $C>0$,
$\ell=\ell(m)>0$ and $N=N(m)\in\N$ such that the following holds. Let $\rho,\widetilde\rho\in
\Cinf_c(D,[0,1])$ with $\widetilde\rho=1$ on a neighborhood of $\supp\rho$. Then
\begin{equation}\label{e}
  \|\rho u\|_{H^{2m}}\leq C\,\mathcal N(\rho,\widetilde\rho)\,k^{\ell}
\bigl(\|\widetilde\rho\,\Box_k^m u\|_{L^2}+\|\widetilde\rho\, u\|_{L^2}\bigr),
\qquad u\in\Cinf(D,L^k),  
\end{equation}
where $\mathcal N(\rho,\widetilde\rho)$ denotes the sum of the $\mathscr C^{N}$-norms of $\rho$ and
$\widetilde\rho$. 
\begin{proof}[Proof of Theorem~\ref{Thm_Loc_Off}]
For $K\Subset\Cs^n$, we may assume that $S^{\Lambda}_k(K)\Subset\frac{1}{2}D_k$ by adjustment of the domain and cut-off functions. Recall that  $\frac{1}{2}D_k\Subset\{\chi_k=1\}$. We prove the following two statements. First, for every $N,r\in\N$, there is $C=C(N,r)>0$ such that
\begin{equation}\label{eq:loc1-pf}
\Big|B_k(x,y)-k^{2|\Lambda|}\,\tBk\big(\Sk^{-1}(x),\Sk^{-1}(y)\big)\Big|_{\mathscr{C}^r}
\leq C\,k^{-N}
\end{equation}
for all $(x,y)\in \frac{1}{2}D_k\times \frac{1}{2}D_k$ and all $k$ large. Second, fix $\delta>0$. For every $N,r\in\N$ and every compact set $K\subset X$ there is $C=C(N,r,K)>0$ such that
\begin{equation}\label{eq:loc2-pf}
|B_k(x,y)|_{\mathscr{C}^r}\leq C\,k^{-N}
\end{equation}
whenever one of $x,y$ lies in $\frac{1}{2}D_k$, the other lies in $K\setminus D_{k,\delta}$, and $k$ is
large. Together, these give Theorem~\ref{Thm_Loc_Off}. Indeed, \eqref{eq:loc1-pf} is statement (a) after
rescaling the coordinates.
For (b), note that for all $\delta>0$, we can adjust  $\epsilon_0>0$ to be small enough such that \[\supp\tchik=:\widetilde{D}_k\subset \frac{1}{4}D_{k,\delta}.\] Note that  
$y$ lies outside $\widetilde{D}_k$ once $|(\Sk)^{-1}(y)|\geq k^{\epsilon_0}$ and $k$ is large. Then, \eqref{eq:loc2-pf} applies.  
  
  We now start proving \eqref{eq:loc1-pf}. By \emph{localized  mild spectral gap} for $\Box_k$ and Lemma \ref{lem:kL}, for any $N\in\N$, there exists $C_{N}$ such that  
 \begin{equation}\label{eq:pff1}
 \begin{split}
   \| \tchik(\Sk)^{-1} \tBk\Sk\chik&-  B_k\,\tchik(\Sk)^{-1} \tBk\Sk\chik\| \\&\leq k^{N_0}\|\Box_k\tchik(\Sk)^{-1} \tBk\Sk\chik\|\leq C_{N}k^{-N}.
 \end{split}
 \end{equation}Combining this fact with \eqref{eq:133}, we have
\begin{equation*}\begin{split}
  \tchik(\Sk)^{-1} \tBk\Sk\chik&\equiv  B_k\,\tchik(\Sk)^{-1} \tBk\Sk\chik\\& \equiv B_kf(k^{-\frac{2}{n}|\Lambda|}\Box_{h^k,\omega_k})\chi_k\\&= B_k\chi_k.   \end{split}
\end{equation*}We set $R_k:=\tchik(\Sk)^{-1} \tBk\Sk\chik-B_k\chi_k\equiv 0$. To improve the regularity, by Lemma \ref{lem:kL}, \begin{equation}\label{432}
    (\Box_k)^m R_k\equiv0,\quad m\in\N.
\end{equation}
We set $\rho_k:X\to [0,1]$ to be a sequence of smooth cut-off functions such that $\rho_k=1$ on $\frac{1}{2}D_k$ and $\supp\rho_k\Subset \{\chi_k=1\}$. To prove $\eqref{eq:loc1-pf}$, we consider \begin{equation*}
    \rho_k\left((\Sk)^{-1}\tBk(\Sk)-B_k\right)\rho_k=\rho_kR_k\rho_k\equiv 0.
\end{equation*}By elliptic estimate \eqref{e} and \eqref{432}, for $m\in\N$ \begin{equation*}
    \|\rho_kR_k\rho_k \|_{L^2\to H^{2m}}=O(k^{-\infty}).
\end{equation*}By duality, $\|\rho_kR_k\rho_k\|_{H^{-2m}\to L^2}=\|\rho_kR_k^*\rho_k\|_{L^2\to H^{2m}}$.
Let $\widetilde\rho_k\in \Cinf_c(X,[0,1])$ satisfy $\widetilde\rho_k=1$ near $\supp\rho_k$ and
$\supp\widetilde\rho_k\Subset\{\chi_k=1\}$. Since
$R_k^*=\chi_k(\Sk)^{-1}\tBk\Sk\tchik-\chi_kB_k$ and $\chi_k=1$ near $\supp\widetilde\rho_k$, we have
$\widetilde\rho_k\Box_k^m R_k^*\rho_k=0$. Hence, by the elliptic estimate and $R_k^*\equiv0$,
\[
\|\rho_kR_k^*\rho_k\|_{L^2\to H^{2m}}
\leq C_mk^{N_m}\bigl(\|\widetilde\rho_k\Box_k^mR_k^*\rho_k\|+\|\widetilde\rho_kR_k^*\rho_k\|\bigr)
=O(k^{-\infty}).
\] By the  standard interpolation theorem for Sobolev spaces, we have \begin{equation*}
    \|\rho_kR_k\rho_k\|_{H^{-m}\to H^{m}}\leq \|\rho_kR_k\rho_k\|_{H^{-2m}\to L^2}^{1/2}\cdot \|\rho_kR_k\rho_k\|_{L^2\to H^{2m}}^{1/2}=O(k^{-\infty}).
\end{equation*} By Sobolev inequality \eqref{s} and repeating the arguments in \eqref{e1}--\eqref{e2}, we prove \eqref{eq:loc1-pf} by noting that the Schwartz kernel of $\left(\Sk\right)^{-1}\tBk\Sk$ is given by \[
k^{2|\Lambda|}\tBk\big((\Sk)^{-1}(z), (\Sk)^{-1}(w)\big),
\]where the factor $k^{2|\Lambda|}$ arises from the change of volume form when changing variables.

To prove the off-diagonal decay \eqref{eq:loc2-pf}, let
$\widehat{\rho}_k\in \Cinf_c(X,[0,1])$
 be a family of cut-off functions such that $\widehat{\rho}_k=1$ on $K\setminus D_{k,\delta}$ and $\operatorname{supp}\widehat{\rho}_k
\cap
\frac{1}{2}{D}_{k,\delta}
=
\varnothing$.
We require the $\Cinf$-seminorms of $\widehat{\rho}_k$ grow at most polynomially in $k$. By the Sobolev inequality, it is sufficient to prove that, for every
\(m\in\mathbb N\),
\begin{equation}\label{eq:suff}
\left\|
\widehat{\rho}_kR_k\rho_k
\right\|_{H^{-m}\to H^{m}}
=
O(k^{-\infty}).
\end{equation}
Choose a cut-off function \(\widetilde{\widehat{\rho}}_k\in \Cinf_c(X,[0,1])\)
such that $\widetilde{\widehat{\rho}}_k=1$ on a neighbourhood of $\operatorname{supp}\widehat{\rho}_k$, and
$\operatorname{supp}\widetilde{\widehat{\rho}}_k
\cap
\widetilde{D}_k
=
\varnothing$. By the local elliptic estimate, for $m\in\N$,
\[
\left\|
{\widehat{\rho}}_kR_k\rho_k
\right\|_{L^2\to H^{4m}}
\leq
C_mk^{N_m}
\left(
\left\|
\widetilde{\widehat{\rho}}_k\Box_k^{2m}R_k\rho_k
\right\|_{L^2\to L^2}
+
\left\|
\widetilde{\widehat{\rho}}_kR_k\rho_k
\right\|_{L^2\to L^2}
\right).
\]
By \eqref{432} and the fact that \(R_k\equiv0\), we obtain
\[\left\|\widehat{\rho}_kR_k\rho_k
\right\|_{L^2\to H^{4m}}=O(k^{-\infty}).
\]
By duality,
$\left\|
\widehat{\rho}_kR_k\rho_k
\right\|_{H^{-4m}\to L^2}
=
\left\|
\rho_kR_k^*\widehat{\rho}_k
\right\|_{L^2\to H^{4m}}.
$
Let $\widetilde{\rho}_k\in \Cinf_c(X,[0,1])$ such that $\widetilde{\rho}_k=1$ on a neighborhood of $\operatorname{supp}\rho_k$ and $\operatorname{supp}\widetilde{\rho}_k
\Subset
\{\chi_k=1\}$.  The elliptic estimate gives
\[
\left\|
\rho_kR_k^*\widehat{\rho}_k
\right\|_{L^2\to H^{4m}}
\leq
C_mk^{N_m}
\left(
\left\|
\widetilde{\rho}_k\Box_k^{2m}
R_k^*\widehat{\rho}_k
\right\|_{L^2\to L^2}
+
\left\|
\widetilde{\rho}_kR_k^*\widehat{\rho}_k
\right\|_{L^2\to L^2}
\right).
\]
Since $\operatorname{supp}\widehat{\rho}_k\cap\operatorname{supp}\tchik=\varnothing$, $R_k^*\widehat{\rho}_k=-\chi_kB_k\widehat{\rho}_k$ by definition. Moreover, since
\(\chi_k=1\) on a  neighbourhood of
\(\operatorname{supp}\widetilde{\rho}_k\), we have
\[
\widetilde{\rho}_k\Box_k^{2m}
\chi_kB_k\widehat{\rho}_k
=
\widetilde{\rho}_k\Box_k^{2m}
B_k\widehat{\rho}_k
=
0.
\]
Together with $R_k^*\equiv0$ and the above arguments, we have  $\left\|
\rho_kR_k^*\widehat{\rho}_k
\right\|_{L^2\to H^{4m}}
=O(k^{-\infty})$. Therefore,
\[
\left\|
\widehat{\rho}_kR_k\rho_k
\right\|_{H^{-4m}\to L^2}
=
O(k^{-\infty}).
\]
We then obtain \eqref{eq:suff} by interpolating again.

\end{proof}
\section{Full asymptotic expansion: Proof of Theorem \ref{Thm_Full expansion}}\label{sec_6}
In this section, we study the full asymptotic expansion of
$\tBk(z,w)$ on $\Cs^n$ in the quasihomogeneous cases. We then apply the localization theorem
(see Theorem~\ref{Thm_Loc_Off}) to derive the expansion of $B_k(x,y)$
near a given point $p\in X$, under the localized mild spectral gap and the model gap condition. Our argument follows the method developed in \cite{hsiao2022bergman}, with the symbol spaces adapted to the present setting.

We say that a symbol $a(z,w) \in \Cinf(\Cs^n \times \Cs^n)$ belongs to the space $\mathscr{F}$ if, for every $r, N \in \N$, there exist constants $C = C(r,N) > 0$ and $\ell = \ell(r) > 0$ such that\[\lvert a(z,w) \rvert_{\mathscr{C}^r} \leq C \frac{(1+\lvert z \rvert+\lvert w \rvert)^{\ell}}{1+\lvert z-w \rvert^N}, \]for all $(z,w)\in\Cs^n\times\Cs^n$.

For $a\in\mathscr F$, we define the associated integral operator $\operatorname{Op}(a):\mathscr{S}(\Cs^n)\to \mathscr{S}(\Cs^n)$ by 
\[
    \operatorname{Op}(a)u(z)
    :=
    \int_{\Cs^n}a(z,w)u(w)\,dV_{\omega_0}(w),
    \qquad
    u\in\mathscr{S}(\Cs^n),
\]where $\mathscr{S}(\Cs^n)$ is the Schwartz space of $\Cs^n$. 
We denote the class of such operators by
$\operatorname{Op}(\mathscr F)$, and write
$\sigma(\operatorname{Op}(a))=a$.

If \(a,b\in\mathscr F\), then
\[
    a\# b(z,w)
    :=
    \int_{\Cs^n}a(z,\xi)b(\xi,w)\,dV_{\omega_0}(\xi)
\]
is absolutely convergent and belongs to $\mathscr F$. This follows from the estimates by taking the
off-diagonal decay orders sufficiently large. Moreover,
\[
    \operatorname{Op}(a)\circ\operatorname{Op}(b)
    =
    \operatorname{Op}(a\# b).
\]


We say a $k$-dependent family of symbols $a_k(z,w) \in \Cinf(\Cs^n \times \Cs^n)$ belongs to the space $k^{m}\mathscr{F}$ for some $m \in \R$ if, for every $r, N \in \N$, there exist constants $C = C(m,r,N) > 0$, $\ell = \ell(m,r) > 0$, and $k_0 \in \N$ such that\begin{equation}\label{Fk}
    \lvert a_k(z,w) \rvert_{\mathscr{C}^r} \leq C k^{m} \frac{(1+\lvert z \rvert+\lvert w \rvert)^{\ell}}{(1+\lvert z-w \rvert)^{N}},
\end{equation} for all $(z,w) \in \Cs^n \times \Cs^n$ and $k\geq k_0$. We define the corresponding space of operators $k^{m}\operatorname{Op}(\mathscr{F})$ in the same manner. If $P_k\in k^{m_p}\operatorname{Op}(\mathscr F)$ and $Q_k\in k^{m_q}\operatorname{Op}(\mathscr F)$, then $P_k\circ Q_k\in k^{m_p+m_q}\operatorname{Op}(\mathscr F)$.
Equivalently,
\begin{equation}\label{eq:sharp_k}
    k^{m_p}\mathscr F\# k^{m_q}\mathscr F
    \subset
    k^{m_p+m_q}\mathscr F.
\end{equation}

\begin{definition}\label{def:expansion}
    Let $\{m_j\}_{j\geq 0} \subset \R$ be decreasing, with $m_0 = m$ and $m_j \searrow -\infty$. We say that a $k$-dependent family $a_k\in k^{m}\mathscr{F}$ has asymptotic expansion $a_k \sim \sum_{j=0}^{\infty} k^{m_j}\, \widehat{a}_j$, where $\widehat{a}_j \in \mathscr{F}$ if
    \begin{equation*}
        a_k - \sum_{j=0}^{N} k^{m_j}\, \widehat{a}_j \in k^{m_{N+1}}\mathscr{F},
        \quad \text{for all } N \in \N .
    \end{equation*}
\end{definition}

Such expansions behave well under composition.

\begin{lemma}\label{lem:symbol_composition}
    Let $p_k \in k^{m_{1,0}}\mathscr{F}$ and $q_k \in k^{m_{2,0}}\mathscr{F}$ have the expansions $p_k \sim \sum_{j\geq 0} k^{m_{1,j}}\widehat{p}_j$ and $q_k \sim \sum_{j'\geq 0} k^{m_{2,j'}}\widehat{q}_{j'}$. List the set $\{\, m_{1,j} + m_{2,j'} : j, j' \in \N_0 \,\}$ in decreasing order as $\{m_{3,l}\}_{l\geq 0}$, and put
    \begin{equation*}
        (p\#q)_l := \sum_{\substack{j,j' \in \N_0 \\ m_{1,j}+m_{2,j'} = m_{3,l}}} \widehat{p}_{j} \# \widehat{q}_{j'} \; \in \mathscr{F} .
    \end{equation*}
    Then
    \begin{equation*}
        p_k \# q_k \sim \sum_{l=0}^{\infty} k^{m_{3,l}}\, (p\#q)_l .
    \end{equation*}
\end{lemma}

\begin{proof}
   Fix $L \in \N$. We may choose $N$ large enough such that $ m_{1,0} + m_{2,N+1} \leq m_{3,L+1}$, $m_{1,N+1} + m_{2,0} \leq m_{3,L+1}$. Set $P_{N,k} := \sum_{j\leq N} k^{m_{1,j}}\widehat{p}_j$ and $Q_{N,k} := \sum_{j'\leq N} k^{m_{2,j'}}\widehat{q}_{j'}$, and write $p_k := P_{N,k} + r_{1,k}$ and $q_k := Q_{N,k} + r_{2,k}$.  This gives
    \begin{equation*}
        p_k \# q_k
        = \sum_{j,j'=0}^{N} k^{m_{1,j}+m_{2,j'}}\, \widehat{p}_j \# \widehat{q}_{j'}
        + P_{N,k} \# r_{2,k} + r_{1,k} \# Q_{N,k} + r_{1,k} \# r_{2,k} .
    \end{equation*}
    The last three terms are in $k^{m_{3,L+1}}\mathscr{F}$ by \eqref{eq:sharp_k} and the choice of $N$. We split the double sum and complete the proof.
\end{proof}
We now consider the context of Theorem \ref{Thm_Full expansion} and make the quasihomogeneous assumption, which means that $\phimodk=\phimod$ is independent of $k$.
Recall the construction of metrics \eqref{eq:sca_metric} and define
\begin{equation}\label{1020}
    V_k(z):=e^{-(\tphik-\phimod)(z)}\sqrt{\frac{dV_{\tomegak}}{dV_{\omega_0}}}=e^{-(\tphik-\phimod)(z)}\beta_k^{-\frac12}(z).
\end{equation}Clearly, $V_k=1$ for all $|z|\geq k^{\epsilon_0}$. By the metric expansions \eqref{-2} and \eqref{-1}, for any $N,r\in\N$, there exists $C=C(N,r)>0$ such that \begin{equation}\label{eq:conn1}
   \Big|V_k(z)-1-\sum_{j=1}^Nk^{-d_j}v_j(z)\Big|_{\mathscr{C}^r}\leq C_{N,r}k^{-c_{N}},\quad \text{for all }\,|z|\leq \frac{1}{2}k^{\epsilon_0},
\end{equation}where $v_j$'s are polynomials. If $d_j\notin O^{\even}_{\Lambda}$, $v_j$ is odd. Similarly, for the inverse of $V_k$, 
\begin{equation}\label{eq:conn2}
   \Big|V_k^{-1}(z)-1-\sum_{j=1}^Nk^{-d_j}\hat v_j(z)\Big|_{\mathscr{C}^r}\leq C_{N,r}k^{-c_{N}},\quad \text{for all }\,|z|\leq \frac{1}{2}k^{\epsilon_0},
\end{equation}where $\hat v_j$ is odd if $d_j\notin \mathcal{O}^{\even}_\Lambda$.

By Theorem \ref{Thm_off_B_mod} and Theorem \ref{Thm_off_B}, we have  \begin{equation*}
    \sBmod(z,w)\in\F \quad ;\quad \tsBk(z,w)\in\F.
\end{equation*}We now define an operator $\hBk\in\Op(\mathscr{F})$ whose symbol is \[
\hBk(z,w)=V_k(z)\sBmod(z,w) V^{-1}_k(w).
\] By direct computation, $\tDk\hBk=0$. Hence, we have $\tsBk\hBk=\hBk$. Let $\hBk^*$ denote the adjoint of $\hBk$ with respect to the standard inner product  $(\cdot|\cdot)_{\omega_0}$. Then, \begin{equation*}
\hBk^*\,\tsBk=\hBk^*.
\end{equation*}  Similarly, since $\Dk V_k^{-1}\tsBk=0$, we have $\hBk \tsBk:\mathscr{S}(\Cs^n)\to \mathscr{S}(\Cs^n)$ satisfying   \[\hBk \tsBk =V_k\cdot \sBmod \cdot V_k^{-1}\cdot \tsBk =\tsBk.\] Denote  $\mathcal{Q}_k:=\hBk-\hBk^*$. We have $(\Id-\mathcal{Q}_k)\tsBk=\hBk^*$. Then, for $N\in\N$, by applying $\sum_{j=0}^N\mathcal{Q}_k^j$ on both sides, we get 
\begin{equation}\label{eq:exp_Q}
(\Id-\mathcal{Q}_k^{N+1})\tsBk=(\Id+\mathcal{Q}_{k}+\cdots+\mathcal{Q}_{k}^{N})\hBk^*.    \end{equation}
Since $\sBmod$ is self-adjoint with respect to $(\cdot|\cdot)_{\omega_0}$, it follows that $\hBk^*
=
\widetilde  V_k^{-1}\sBmod V_k$.
 Hence, the Schwartz kernel $\hBk^*(z,w)$ of $\hBk^*$ is given by \begin{equation}\label{hBk}
\hBk^*(z,w)=V_k^{-1}(z)\sBmod(z,w)V_k(w).
\end{equation}Also, the Schwartz kernel of $\mathcal{Q}_k$ is given by \begin{equation}\label{eq:07161}
  \mathcal{Q}_k(z,w)=\sBmod(z,w)\Big[V_k(z)\cdot V_k^{-1}(w)-V_k^{-1}(z)V_k(w)\Big]. 
\end{equation} 

\begin{lemma}\label{lem:exp} We have the formal asymptotics: 
    \begin{align}\label{lem:eq1}
        \hBk^*(z,w)&\sim \sBmod(z,w)+\sum_{j=1}^\infty k^{-d_j}\sBmod(z,w)b_j(z,w);\\ \label{lem:eq2}
     \mathcal{Q}_k(z,w)&\sim \sum_{j=1}^\infty k^{-d_j}\sBmod(z,w)q_j(z,w),
        \end{align} where $b_j(z,w)$ and $q_j(z,w)$ are odd polynomials if $d_j\notin \mathcal{O}^{\even}_\Lambda$.
    \end{lemma}
\begin{proof}Take \[b_j(z,w):=\sum_{\substack{j_1,j_2\in\N_0\\d_{j_1}+d_{j_2}=d_j}}\widehat v_{j_1}(z) v_{j_2}(w).\] Let $S_{N,k}:=\sum_{j=0}^Nk^{-d_j}v_j(z)$ and $\widehat S_{N,k}:=\sum_{j=0}^Nk^{-d_j}\widehat v_j(z)$. By \eqref{hBk}, to show \eqref{lem:eq1}, it is sufficient to show that \begin{equation}\label{claim}
   P_{N,k}:=V_k^{-1} \sBmod V_k-\widehat S_{N,k}\sBmod  S_{N,k}\in k^{-\mu_N}\Op(\F), 
\end{equation}where $\mu_N\to \infty$ as $N\to\infty$. 

Recall that $\chi\in\Cinf_c(\Cs^n,[0,1])$ is smooth with $\chi(z)=1$ when $|z|\leq 1/2$ and $\chi(z)=0$ when $|z|\geq 1$. Define $\chi_k(z,w):=\chi(2zk^{-{\epsilon_0}})\chi(2wk^{-{\epsilon_0}})$. By \eqref{eq:conn1}, \eqref{eq:conn2} and $\sBmod\in\F$, it is straightforward to see that for $N\in\N$,  \begin{equation}\label{1247}
    \chi_k(z,w)P_{N,k}(z,w)\in k^{-c_N}\F.
\end{equation}To control $P_{N,k}$ on $\operatorname{supp}(1-\chi_k)$, where $\max\{|z|,|w|\}\geq \tfrac14 k^{\epsilon_0}$, for any $m\in\N$, we can choose a large enough $\ell=\ell(m)$ such that 
   \[
 (1+\lvert z\rvert+\lvert w\rvert)^{\ell}\geq k^{m},\quad \text{for all } (z,w)\in \supp(1-\chi_k).
 \] We can use it to control the \(\Cr\)-norm of \(P_{N,k}(z,w)\) when \((z,w)\in\supp (1-\chi_k)\). For any \(N,r,m\in\N\), we can choose a large enough \(\ell=\ell(r,N,m)\) and \(C=C(r,N,m)>0\) such that \begin{equation*}
\lvert V_k(z)V^{-1}_k(w)-S_{N,k}(z)\widehat S_{N,k}(w)\rvert_{\Cr}\leq Ck^{-m}(1+\lvert z\rvert+\lvert w\rvert)^{\ell},\quad \text{for } (z,w)\in\supp (1-\chi_k(z,w)).
\end{equation*}Combining this with the fact that  \(\sBmod\in\F\), we have for all \(m\in \N\), \[(1-\chi_k(z,w))P_{N,k}(z,w)\in k^{-m}\F.\]  By the above arguments and \eqref{1247}, we obtain \eqref{claim}, and hence complete the proof of \eqref{lem:eq1}. In the same manner, the proof of \eqref{lem:eq2} can be derived from \eqref{eq:07161} and the similar process above. The last statement follows from the combination of polynomials.
\end{proof}

\begin{proof}[Proof of Theorem \ref{Thm_Full expansion}]By the unitary identification \eqref{unitary_t}, and since $\tBk(z,w)$ is the Schwartz kernel with respect to $dV_{\tomegak}$ while $\tsBk(z,w)$ is taken with respect to $dV_{\omega_0}$, where $\beta_k=dV_{\omega_0}/dV_{\tomegak}$, we have
\[
\tBk(z,w)=e^{\tphik(z)}\beta^{\frac{1}{2}}_k(z)\tsBk(z,w) e^{-\tphik(w)}\beta^{\frac{1}{2}}_k(w).
\]Note that $dV_{\tomegak}=k^{2|\Lambda|}S_k^*dV_\omega$ for $|z|<\frac12k^{\epsilon_0}$, which accounts for the factor $k^{2|\Lambda|}$ in Theorem~\ref{Thm_Loc_Off}. Thanks to the localization property (see Theorem \ref{Thm_Loc_Off}), it is sufficient to study the full expansion of $\tsBk(z,w)$ on a bounded domain in $\Cs^n$. Denote \[
\Qk^{\#j}:=\underbrace{\Qk\#\Qk\#\cdots\#\Qk}_{j-\text{times}}.
\]
From Lemma \ref{lem:symbol_composition} and \eqref{eq:exp_Q}, one derives: \begin{equation}\label{eq:Qk_exp}
    \tsBk \sim \sum_{j=0}^\infty(\Qk)^{\#j}\#(\hBk)^*.
\end{equation}
By Lemma \ref{lem:symbol_composition} again, we substitute the expansions of $\Qk$ and $(\hBk)^*$ given in Lemma \ref{lem:exp}, which are expressed in terms of $\sBmod$ and the polynomials $b_j$ and $q_j$, into the asymptotic sum above. This yields an expansion of the following form:\begin{equation}\label{eq:exp_pf}
    \tsBk\sim \sBmod+\sum_{j=1}^\infty k^{-d_j}\Phi_j(z,w),
\end{equation} for some $\Phi_j(z,w)\in \F$. For $|z|<\frac12k^{\epsilon_0}$, we have $\beta_k(z)=v(0)/v(S^\Lambda_k(z))$, so the factors $e^{\pm(\tphik-\phimod)}\beta_k^{\pm\frac12}$ admit expansions in powers $k^{-d_j}$ of the same type as \eqref{eq:conn1}, with leading term $1$. Since $e^{\phimod(z)}\sBmod(z,w)e^{-\phimod(w)}=\Bmod(z,w)$ by \eqref{123}, we obtain the desired full asymptotic expansion for $\tBk(z,w)$ with leading term $\Bmod(z,w)$, and hence complete the proof of the first statement in Theorem \ref{Thm_Full expansion}.


To prove the second statement in Theorem \ref{Thm_Full expansion}, we assume that $\phimod$ is even and aim to show that $\Phi_{j}$ is an odd function if $d_j\notin \mathcal{O}^{\even}_\Lambda$. Note that the model Bergman
kernels are even in the sense that 
\[
\sBmod(z,w)=\sBmod(-z,-w).
\]
We first make a simple observation.  Let $P(\xi)$ be an even polynomial in the sense that $P(\xi)=P(-\xi)$. By the change of variables
$\xi\mapsto-\xi$,
\begin{multline*}
(\sBmod P\sBmod)(z,w)=\int_{\Cs^n}\sBmod(z,\xi)P(\xi)\sBmod(\xi,w)\,dV_{\omega_0}(\xi)\\
=\int_{\Cs^n}\sBmod(-z,-\xi)P(-\xi)\sBmod(-\xi,-w)\,dV_{\omega_0}(\xi)=(\sBmod P\sBmod)(-z,-w).
\end{multline*}
Hence, $\sBmod P\sBmod$ is even. In the same way, it is odd when $P$ is odd. More generally, let
$P(z,\xi_1,\dots,\xi_p,w)$ be an even (resp.\ odd) polynomial. Then
\[
\int\cdots\int P(z,\xi_1,\dots,\xi_p,w)\,\sBmod(z,\xi_1)\sBmod(\xi_1,\xi_2)\cdots\sBmod(\xi_p,w)
\,dV_{\omega_0}(\xi_1)\cdots dV_{\omega_0}(\xi_p)
\]
is even (resp. odd).

From \eqref{eq:Qk_exp}, \eqref{eq:exp_pf} and Lemma \ref{lem:exp}, we observe that for $j_0\in\N$,
\begin{equation*}
\Phi_{j_0}(z,w)=\sum_{\ell\in\N_0}\
\sum_{\substack{(i_1,\cdots,i_\ell)\in\N^\ell,\ i_0\in\N_0\\ d_{i_1}+\cdots+d_{i_\ell}+d_{i_0}=d_{j_0}}}
(\sBmod q_{i_1})\#\cdots\#(\sBmod q_{i_\ell})\#(\sBmod b_{i_0}),
\end{equation*}
where we use the convention $d_0:=0$ and $b_0:=1$. For each monomial in $q_i$ and $b_i$, which appears with the factor $k^{-d_i}$, there is an admissible $\mathbf n$ with $d(\mathbf n)=d_i$ such that the total degree of the monomial is odd if and only if $\ell(\mathbf n)$ is odd. Consequently, if  $d_{j_0}\notin \mathcal O^{\even}_\Lambda$, then any $(i_0,i_1,\cdots,i_\ell)$ with $d_{i_1}+\cdots+d_{i_\ell}+d_{i_0}=d_{j_0}$ must have the property that
\begin{equation*}
q_{i_1}(z,\xi_1)\,q_{i_2}(\xi_1,\xi_2)\cdots q_{i_\ell}(\xi_{\ell-1},\xi_\ell)\, b_{i_0}(\xi_\ell,w)
\end{equation*}
is an odd polynomial. Hence, $\Phi_{j_0}(z,w)$ is an odd function if $d_{j_0}\notin \mathcal{O}^{\even}_\Lambda$. This completes the proof.
\end{proof}
\section{On good finite type: Theorem \ref{Thm_de}}\label{sec_de}

Let $(L,h)$ be a holomorphic line bundle with smooth Hermitian metric
satisfying $\ii\Theta_h(L)\geq0$. Assume that $p$ is of good finite
type in the sense of Definition~\ref{def_finite}. Let
$\mathcal M=(m_1,\ldots,m_n)$ and set
$\Lambda=(\frac1{2m_1},\ldots,\frac1{2m_n})$.
In the corresponding coordinates and local frame, we write
\[
\varphi_{\mathcal M}
=\sum_{j=1}^nP_j(z_j,\bar z_j)+Q(z,\bar z).
\]
By restricting $\varphi_\mathcal{M}$ to the $j$-th coordinate axis and using the
semipositivity of $\ii\Theta_h(L)$,  we see that $\Delta P_j\geq 0$, where $\Delta=4\partial_z\partial_{\bar z}$ denotes the standard Laplacian.  Moreover, since $P_j$ is nonzero and contains no nonconstant harmonic
terms, $\Delta P_j\not\equiv0$. Here, $\Delta P_j$ is a homogeneous
polynomial of degree $2m_j-2$. 
\begin{proposition}\label{prop_positive_nearby}
If $p$ is of good finite type with respect to $h$, then $p$
belongs to the closure of the positive-curvature locus. More precisely,
there exists a sequence $x_\nu\to p$ such that
\[
\ii\Theta_h(L)(x_\nu)>0.
\]
\end{proposition}
\begin{proof}
 For each $j=1,\dots,n$, choose $\hat z_j\in\Cs$ such that $\Delta P_j(\hat z_j)>0$, and set $\hat z=(\hat z_1,\ldots,\hat z_n)$. Note that $\phimodk(z)=\sum_{j=1}^nP_j(z_j)+kQ(S_k^{\Lambda}(z))$. Since $Q$ is a plurisubharmonic polynomial, we have \[\ii\partial\db\phimodk\Big|_{z=\hat z}=\ii\partial\db(\sum_{j=1}^nP_j+k\cdot Q\circ S^\Lambda_k)\Big|_{z=\hat z}>0.\]  By the remainder estimate \eqref{re}, $k\varphi(S_k^\Lambda z)-\varphi_k^\Lambda(z)
\longrightarrow0$ in the $\mathscr C^2$-topology on compact subsets. It follows that, for all
sufficiently large $k$, \[\ii\partial\db k\vphi(S^\Lambda_k(z))|_{z=\hat z}>0.\] Thus, setting $x_k=S_k^\Lambda(\hat z)$, we have $x_k\to p$ and $\ii\Theta_h(L)(x_k)>0$ for all sufficiently  large $k$.
\end{proof}
\begin{corollary}\label{cor_good_finite_big}
Assume that $X$ is compact. If $p$ is of good finite type, then
$L$ is big.
\end{corollary}

\begin{proof}

By Proposition~\ref{prop_positive_nearby}, the semipositive curvature form
$i\Theta_h(L)$ is positive definite at some point of $X$.
Hence $L$ is big by Siu's criterion (see \cite[p.175]{Siu1985} or 
\cite[Theorem~2.2.27]{ma2007holomorphic}).

\end{proof}
Assume further that  $p\in\operatorname{Amp}(L)$ (see Definition \ref{ample}).
After shrinking $D$ if necessary, there exists a singular
Hermitian metric $h_{\mathrm{sing}}$ on $L$ with analytic
singularities such that
$i\Theta_{h_{\mathrm{sing}}}(L)\geq\varepsilon\omega$
in the sense of currents for some $\varepsilon>0$,
and $h_{\mathrm{sing}}$ is smooth near $\overline D$. Then, we can apply \cite[Lemma~8.1]{hsiao2014asymptotics},
whose proof relies on H\"ormander-type $L^2$ estimates
and an auxiliary singular Hermitian metric
with a suitable positive curvature lower bound. There exist $C_D>0$ and $N\in\mathbb N$ such that,
for every
$\eta\in\Cinf_c(D,\Lambda^{0,1}T^*X\otimes L^k)$
with $\bar\partial_k\eta=0$, there exists
$v\in L^2(X,L^k)$ satisfying $\bar\partial_kv=\eta$ and
\[
\|v\|_{h^k,\omega}^2
\leq C_Dk^N\|\eta\|_{h^k,\omega}^2,
\]for all sufficiently large $k$. For $u\in\Cinf_c(D,L^k)$, applying this estimate to
$\eta=\bar\partial_k u$ and using the minimality of
$(\Id-B_k)u$, we obtain
\[
\begin{aligned}
\|(\Id-B_k)u\|^2
\leq C_Dk^N\|\bar\partial_k u\|^2&=C_Dk^N(\Box_k u,(\Id-B_k)u)\\
&\leq C_Dk^N\|\Box_k u\|\,\|(\Id-B_k)u\|.
\end{aligned}
\]
Consequently,
$
\|(\Id-B_k)u\|\leq C_Dk^N\|\Box_k u\|$. We formalize the result as follows.

\begin{proposition}\label{prop_ample_local_gap}
Assume that $X$ is compact and that $p$ is of good finite type.
If
\[
p\in\operatorname{Amp}(L),
\]
then, after shrinking $D\ni p$, the Kodaira Laplacian $\Box_k$ satisfies the localized mild spectral gap condition on $D$.
\end{proposition}

In the remaining section, we aim to prove the model gap condition. We first establish the existence of suitable cut-off functions as follows:

\begin{lemma}\label{lem_p}
    Let $\widetilde{P}(z,\bar z)$ be a non-negative real-valued homogeneous polynomial on $\Cs$. Assume that $\widetilde{P}\not\equiv 0$. Then, there exists  $f\in \Cinf(\Cs)\cap L^\infty(\Cs)$ such that $\widetilde{P}+\Delta f\geq c$ on $\Cs$ for some $c>0$. 
\end{lemma}
\begin{proof} Since $\widetilde{P}\geq 0$ is a homogeneous polynomial on $\R^2$ and its restriction to the unit circle has only finitely many zeros, we have 
\[
    \{\widetilde{P}=0\}
    =
    \bigcup_{j=1}^{n_0}L_j,
\]
where $L_j$'s are lines passing through the origin. For each $L_j$, we denote by $d(\cdot,L_j)$ the distance function to $L_j$. Note that $d(\cdot,L_j)^2$ is  a smooth subharmonic function such that  $\Delta d(\cdot,L_j)^2\geq 1$ for all $z\in\Cs$. Let $\rho\in\Cinf_c(\R)$ such that $\rho=1$ on $[-\frac{1}{2},\frac{1}{2}]$ and $\supp\rho\subset[-1,1]$. We may choose a small enough $\epsilon>0$ such that there exist $c>0$ and a compact set $K\Subset\Cs$ satisfying \[ \widetilde{P}+\Delta \left(\epsilon \sum_{j=1}^{n_0}\rho(d(\cdot,L_j)^2)\cdot d(\cdot,L_j)^2\right)\geq c, \]for all $z\in\Cs\setminus K$. In fact, the homogeneity of $\widetilde P$ gives a uniform positive lower
bound for $\widetilde P$ on the parts of the strips
$\frac{1}{2}\leq d(z,L_j)^2\leq 1$
lying outside an appropriate compact set. We take advantage of this fact and choose
$\epsilon>0$ sufficiently small so that the negative contribution of
$ \Delta\left(\rho\bigl(d(\cdot,L_j)^2\bigr)d(\cdot,L_j)^2\right)$ on these strips is absorbed by the uniform positivity of $\widetilde P$. Next, we pick a large enough $C>0$ such that \[ \widetilde{P}+\Delta \left(\epsilon \sum_{j=1}^{n_0}\rho(d(\cdot,L_j)^2)\cdot d(\cdot,L_j)^2+C\log(1+|z|^2)\right)\geq c, \] for all $z\in K$. By direct computation, we can choose a large enough $R>0$ such that \[ \widetilde{P}+\Delta \left(\underbrace{\epsilon \sum_{j=1}^{n_0}\rho(d(\cdot,L_j)^2)\cdot d(\cdot,L_j)^2+C\rho(\frac{|z|^2}{R})\log(1+|z|^2)}_{:=f}\right)\geq \frac{c}{2}, \] for all $z\in\Cs$. Since $f$ is a bounded smooth function, it completes the proof.

\end{proof}
\begin{proposition}[Model gap for good finite type]\label{model_gap_finite_typ}\label{prop_decouple}
    If $p$ is of good finite type, the model gap condition holds in the sense that $\Spec\Box_{\phimodk}\cap(0,c)=\varnothing$, for some $c>0$.
\end{proposition}
\begin{proof}
Taking $\widetilde{P}=\Delta P_j$ in Lemma \ref{lem_p}, we associate with each $P_j$ an adjustment function $f_j$. Set $F=\sum_{j=1}^nf_j\in \Cinf(\Cs^n)\cap L^\infty$.  Then, $\ii\partial\db (\phimodk+F)>c$ uniformly for some $c>0$. Hence, we can apply H\"{o}rmander's $L^2$-estimates; see \cite{hormander1973introduction}, and the compatibility of norms to conclude that for any $(0,1)$-form $\eta\in L^2_{\phimodk}(\Cs^n,\Lambda^{0,1}T^*\Cs^n)$ with $\db \eta=0$, there exists $u\in L^2_{\phimodk}(\Cs^n)$ such that $\db u=\eta$ and   \[
 \|u\|_{\phimodk}\lesssim \|u\|_{\phimodk+F}\leq c^{-1}\|\eta\|_{\phimodk+F}\lesssim \|\eta\|_{\phimodk}.
 \]For $u\in L^2_{\phimodk}(\Cs^n)$ with $u\perp \Ker\Box_{\phimodk}$, we have  $\|u\|_{\phimodk}^2\lesssim\|\db u\|_{\phimodk}^2=(\Box_{\phimodk}u|u)$. By the Cauchy-Schwarz inequality, we deduce $\|u\|_{\phimodk}\lesssim \|\Box_{\phimodk}u\|$, which implies the desired spectral gap.     
\end{proof}
If $p$ is of quasihomogeneous good finite type, then
$\phimodk$ is independent of $k$, and we write
$\phimod:=\phimodk$. The following lemma shows that the leading
coefficient of the asymptotic expansion is strictly positive.
\begin{lemma}\label{lemo}
If $p$ is of quasihomogeneous good finite type, then
$\Bmod(0,0)>0$.
\end{lemma}
\begin{proof}Notice that it is sufficient to find a holomorphic function $h$ such that $h(0)=1$ and  $\int|h|^2 e^{-2\phimod}dV<\infty$. By Lemma \ref{lem_p} , there exists $F\in\Cinf(\Cs^n)\cap L^\infty$ such that $\ii \partial\db (\phimod+F)>c$ uniformly for some $c>0$. Let $\rho\in\Cinf_c(\R)$ such that $\rho=1$ on $[-\frac{1}{2},\frac{1}{2}]$ and $\supp\rho\subset[-1,1]$. One may choose a large enough $R>0$ such that  
    \[
    \ii \partial\db (\underbrace{\phimod+F+\rho(\frac{|z|^2}{R^2})n\log(|z|)}_{:=\widetilde{\vphi}})\geq \frac{c}{2}\quad \text{in the sense of current.}
    \]Let $v\in\Cinf_c(\Cs^n,[0,1])$ such that $v\equiv1$ near $0$. By H\"{o}rmander $L^2$-estimates, there exists
$v'\in L^2_{\widetilde{\vphi}}(\Cs^n)$ such that $\db v'=\db v$ and $
    \|v'\|^2_{\widetilde{\vphi}}\leq \frac{2}{c}\|\db v\|^2_{\widetilde{\vphi}}$. Moreover, since $\widetilde{\vphi}(z)=n\log|z|+O(1)$ as $z\to 0$, so
$\widetilde{\vphi}$ has Lelong number $n$ at the origin. Hence,
$v'\in L^2_{\widetilde{\vphi}}$ implies $v'(0)=0$. Observe that  $v'\in L^2_{\phimod}$, and hence $v-v'$ is the desired holomorphic function.  

The construction can be carried out uniformly for compact families of good finite-type model polynomials with fixed multitype.
\end{proof}

\begin{proof}[Proof of Theorem \ref{Thm_de}]
By Proposition  \ref{prop_ample_local_gap} and Proposition \ref{model_gap_finite_typ}, the assumptions
of Theorem \ref{Thm_de} imply the model gap hypothesis and the localized
mild spectral gap hypothesis. Hence, Theorem \ref{Thm_Loc_Off} gives the
localization and off-diagonal decay. Under the additional assumption that $p$ is of quasihomogeneous good
finite type, the quasihomogeneous model gap condition holds. Therefore,
Theorem \ref{Thm_Full expansion} gives the full asymptotic expansion, and
Lemma \ref{lemo} shows that $\Bmod(0,0)>0$.

\end{proof}

\section{Branched coverings:  Theorem \ref{Thm_branched1}}\label{sec_b}
Let $f:X\to Y$ be a generically finite holomorphic map
between compact complex manifolds of dimension $n$,
and let $L$ be an ample line bundle on $Y$.
Then, $f^*L$ is nef and big. Recall that, for a big class $[\alpha]\in H^{1,1}(X)$, Boucksom
\cite[Definition~3.16]{boucksom2004divisorial} defined its non-K\"ahler
locus by
\[
E_{\mathrm{nK}}(\alpha)
:=
\bigcap_{\substack{T\in\alpha\\
                   T\text{ a K\"ahler current}}}
E_+(T),
\] where $E_+(T):=\{x\in X:\nu(T,x)>0\}$ is the positive Lelong locus of $T$.
By \cite[Theorem~3.17(ii)]{boucksom2004divisorial}, the ample locus satisfies
\[
\operatorname{Amp}(f^*L)
=
X\setminus E_{\mathrm{nK}}(c_1(f^*L)).
\]We have the following characterization of the
ample locus of $f^*L$.

\begin{proposition}\label{amplelocus}
Let $f:X\to Y$ and $L$ be as above. Then
\[
\operatorname{Amp}(f^*L)
=
\{x\in X:\dim_x f^{-1}(f(x))=0\},
\]
namely, the locus where $f$ is locally finite.
\end{proposition}
\begin{proof}
Since $f^*L$ is big, $X$ is Moishezon and hence
belongs to Fujiki's class $\mathcal C$.
Moreover, $f^*L$ is nef and
$\int_X c_1(f^*L)^n>0$. Thus, by Collins--Tosatti
\cite[Theorem~1.1]{collins2015kahler},
\[
E_{\mathrm{nK}}(c_1(f^*L))
=
\operatorname{Null}(c_1(f^*L)),
\]
where the null locus of a nef and big class
$[\alpha]\in H^{1,1}$ is defined by
\[
\operatorname{Null}(\alpha)
:=
\bigcup_{\substack{V\subset X\text{ irreducible},\\
                  \dim V=k\geq 1,\\
                  \int_V\alpha^k=0}}V.
\]Now, let $V\subset X$ be an irreducible analytic
subvariety of dimension $k\geq 1$.
By the projection formula and the ampleness of
$L$, we have
\[
\int_V c_1(f^*L)^k
=
\begin{cases}
\deg(f|_V)\displaystyle\int_{f(V)}c_1(L)^k>0,
&\dim f(V)=k,\\[6pt]
0,
&\dim f(V)<k.
\end{cases}
\]
Consequently,
\[
\operatorname{Null}(c_1(f^*L))
=
\bigcup_{\dim f(V)<\dim V}V
=
\{x\in X:\dim_x f^{-1}(f(x))>0\},
\]
where the last equality follows from the
fibre dimension theorem.
\end{proof}
Fix a ramification point $p\in X$, where $f$ is locally finite, such that 
\[
    \operatorname{rank}_{\mathbb C}df_{p}=n-1,
    \qquad y_0:=f(p).
\]  
Let $h$ be a positive Hermitian metric on $L$, and let $e_L$ be a local holomorphic frame of $L$ near $y_0$ with $|e_L|_{h}^2=e^{-2\psi}$. Choose linear coordinates $w=(w_1,w')$ centered at $y_0$ such that $\frac{\partial}{\partial w_2},\cdots,\frac{\partial}{\partial w_n}$ form a basis of $\operatorname{Im}df_p$ which is orthogonal with respect to $\ii\Theta_{h}(L)_{y_0}$, and $\frac{\partial}{\partial w_1}$ spans the orthogonal complement of $\operatorname{Im}df_p$ with respect to $\ii\Theta_{h}(L)_{y_0}$. Since $\ii\Theta_h(L)>0$ at $y_0$, after adjusting the frame, we have
\begin{equation*}
 \psi(w)=\sum_{j=1}^n c_j|w_j|^2+O(|w|^3),\qquad c_j>0 .   
\end{equation*}    
Since $\operatorname{Im}df_p=\{dw_1=0\}$, the differentials of $w_2\circ f,\ldots,w_n\circ f$ at $p$ are linearly independent. Hence, $z':=(w_2\circ f,\ldots,w_n\circ f)$ completes to a holomorphic coordinate system $z=(z_1,z')$ centered at $p$ such that
\begin{equation}\label{fg}
f(z_1,z')=\bigl(g(z_1,z'),z'\bigr).
\end{equation}
Moreover, $dg(0)=0$ by the rank hypothesis. Replacing $w_1$ by $w_1-g(0,w')$, we may assume $g(0,z')\equiv0$. Since $dg(0)=0$, this change  preserves the above normalization of $\psi$. Since $f$ is finite, $g(z_1,0)\not\equiv0$. By the Weierstrass preparation theorem, $g=U\cdot P$ with $U(0)\neq0$ and
\[
    P(z_1,z')=z_1^m+\sum_{j=1}^{m-1}a_j(z')z_1^j,
    \qquad a_j(0)=0,
\]
where $m\in\N$ is given in \eqref{m}.

The pullback frame $f^*e_L$ has weight $\varphi:=\psi\circ f$, which takes the form
\[
    \varphi(z)=c_1|U(z)|^2|P(z_1,z')|^2+\sum_{j=2}^n c_j|z_j|^2
            +O\bigl(|(g(z_1,z'),z')|^3\bigr).
\]
Set $\Lambda:=\bigl(\tfrac1{2m},\tfrac12,\ldots,\tfrac12\bigr)$. Since $a_j(z')=O(|z'|)$ and $U=U(0)+O(|z|)$, every monomial of $a_j(z')z_1^j$ and of $(U-U(0))P$ has $\Lambda$-weight $>\tfrac12$. Hence, $g=UP=U(0)z_1^m+R_g$ with all monomials of $R_g$ of weight $>\tfrac12$. Every monomial of $O\bigl(|(g,z')|^3\bigr)$ has weight $\geq\tfrac32$. We obtain
\begin{equation}\label{mod}
\phimod(z)=c_1|U(0)|^2 |z_1|^{2m}+\sum_{j=2}^n c_j|z_j|^2,
\end{equation}
which is a $\Lambda$-homogeneous polynomial. Then, $p$ is of quasihomogeneous good finite type, with a decoupled model.

To define the leading coefficient intrinsically, we fix the remaining freedom in the coordinates by normalizing $z_1$ so that $\big|\tfrac{\partial}{\partial z_1}\big|_{\omega,p}=1$. Note that $\frac{\partial}{\partial z_1}$ spans $\ker df_p$. We call the resulting $(z,w)$ \emph{normalized adapted coordinates} at $p$ and $y_0$. Since $\ii\Theta_{h}(L)_{y_0}>0$, it induces a Hermitian norm $|\cdot|_{\Theta}$ on $\Lambda^{1,0}T_{y_0}Y$. By construction and $\ii\Theta_{h}(L)=2\ii\partial\db\psi$, we have
\[
    \ii\Theta_{h}(L)_{y_0}
    =
    2\ii c_1\,dw_1\wedge d\overline{w}_1
    +
    2\ii\sum_{j=2}^{n}c_j\,dw_j\wedge d\overline{w}_j .
\]
In particular, $\big|\frac{\partial}{\partial w_1}\big|_{\Theta}^{2}=2c_1$, so $c_1$ measures the curvature of $L$ in the direction orthogonal to $\operatorname{Im}df_{p}$. Moreover,
\[
    g(z_1,0)
    =
    U(0)z_1^m+O(z_1^{m+1}).
\]
Thus, $U(0)\frac{\partial}{\partial w_1}$ can be regarded as the leading normal coefficient of $f$ along the $\omega$-unit vector $\frac{\partial}{\partial z_1}\in\ker df_{p}$. This motivates the following definition.
\begin{definition}\label{Def_a0}
We set
\[
    a_0(p):=2c_1|U(0)|^2=\left|
        U(0)\frac{\partial}{\partial w_1}
    \right|_{\Theta}^{2}.
\]
\end{definition}
\begin{lemma}\label{Lem_a0_intrinsic}
The number $a_0(p)$ does not depend on the choice of normalized adapted
coordinates. It depends only on $f$, on $(L,h)$, and on the value of
$\omega$ at $p$.
\end{lemma}

\begin{proof}
Let $(\widetilde z,\widetilde w)$ be another system of normalized adapted
coordinates, and let
\[
    \widetilde f(\widetilde z)
    =
    \bigl(\widetilde g(\widetilde z),\widetilde z'\bigr),
    \qquad
    \widetilde g=\widetilde U\widetilde P,
\]
be the corresponding representation. Both
$\frac{\partial}{\partial z_1}$ and
$\frac{\partial}{\partial\widetilde z_1}$
span $\ker df_{p}$. Likewise, both
$\frac{\partial}{\partial w_1}$ and
$\frac{\partial}{\partial\widetilde w_1}$
span the curvature-orthogonal complement of
$\operatorname{Im}df_{p}$. Hence,
\[
    \widetilde z_1
    =
    \alpha z_1+T(z')+O(|z|^2),
    \qquad
    \widetilde w_1
    =
    \beta w_1+O(|w|^2),
\]
for some $\alpha,\beta\neq0$ and some linear form $T$. The normalization
of the source coordinates gives $|\alpha|=1$. Comparing the leading normal terms in the two coordinate systems gives
$ \widetilde U(0)
    =
    \beta\alpha^{-m}U(0)$. On the other hand, the curvature coefficient transforms as $\widetilde c_1=|\beta|^{-2}c_1$.
Therefore,
\[\widetilde c_1|\widetilde U(0)|^2
    =|\beta|^{-2}c_1
    \cdot
    |\beta|^2|\alpha|^{-2m}|U(0)|^2 
 =c_1|U(0)|^2.
\]
\end{proof}
\begin{proof}[Proof of Theorem \ref{Thm_branched1}]
By Proposition~\ref{amplelocus}, we know $p\in\operatorname{Amp}(f^*L)$.
We work in normalized adapted coordinates $z=(z_1,z')$, $z'=(z_2,\cdots,z_n)$,
at $p$. In particular, $\frac{\partial}{\partial z_1}$ spans $\ker df_{p}$ and
$\big|\frac{\partial}{\partial z_1}\big|_{\omega,p}=1$. Assume 
$
\omega(p)=\ii\sum_{i,j=1}^n \omega_{i,j}\,dz_i\wedge d\bar z_j$, $G:=(\omega_{i,j})_{i,j=1}^n$, $\omega_{1, 1}=1$ . By \eqref{mod}, the model weight takes the form
\begin{equation}
2\phimod=a_0(p)|z_1|^{2m}+\sum_{j=2}^n 2c_j|z_j|^2,
\end{equation}
where $a_0(p)>0$ and $c_j>0$. Since $p$ is of quasihomogeneous good
finite type, Theorem \ref{Thm_de} gives the full expansion and the
off-diagonal estimate. It remains to compute the leading coefficient
$\Bmod(0,0)$. Recall from Theorem~\ref{Thm_Full expansion} that $\Bmod$ is the Bergman kernel of the Fock space with weight measure $v(0)e^{-2\phimod}dV$, where $\frac{\omega(p)^n}{n!}=v(0)\,dV$ and $v(0)=2^n\det G$. Then, 
\[
\int_{\Cs^n}e^{-2\phimod}dV
=\pi^n\,\Gamma\Big(1+\frac1m\Big)\,a_0^{-1/m}\prod_{j=2}^n(2c_j)^{-1}.
\]
Since $\phimod$ is torus invariant, the
monomials are mutually orthogonal, and hence
\[
\Bmod(0,0)=\Big(v(0)\int_{\Cs^n}e^{-2\phimod}dV\Big)^{-1}
=\frac{a_0^{1/m}\prod_{j=2}^n(2c_j)}{(2\pi)^n\,\Gamma(1+\frac1m)\,\det G}.
\]
Finally, since $\ii\Theta_{f^*h_L}(f^*L)_{p}$ vanishes on
$\ker df_{p}$ and $\omega_{1, 1}=1$, we have
\[
\det_{\omega}\Big(
\ii\Theta_{f^*h_L}(f^*L)_{p}
\big|_{(\ker df_{p})^{\perp_\omega}}
\Big)
=\frac{\prod_{j=2}^n(2c_j)}{\det G},
\]
which proves \eqref{PHI}.

\end{proof}
\begin{remark}\label{rem:higher-corank}
    From the proof, we can see that the localization and off-diagonal decay estimate (see Theorem \ref{Thm_Loc_Off}) still apply if $\rank df<n-1$ but with a more subtle rank condition on the \emph{strata}. For example, if $f=(z_1^{m_1},\cdots,z_n^{m_n})$ and $\psi(w)=H(w,w)+O(|w|^3)$ with $H$ positive definite, then the model polynomial is $\phimod(z)=H(F(z),F(z))$, where $F(z)=(z_1^{m_1},\cdots,z_n^{m_n})$. Choosing $\varepsilon>0$ with $H-\varepsilon\,\mathrm{Id}\geq0$, we can write $\phimod=\varepsilon\sum_{j=1}^n|z_j|^{2m_j}+(H-\varepsilon\,\mathrm{Id})(F,F)$, where the last term is plurisubharmonic. Hence,  the model gap still holds. We can establish off-diagonal decay and full expansion near the \emph{good strata}.
\end{remark}

\begin{proof}[Proof of Proposition \ref{prop.uniform}]

After shrinking the neighborhoods, $f|_M$ is an immersion,
$\ker df_p$ is transverse to $T_pM$, and $f(M)$ is a smooth
hypersurface. Choose holomorphic coordinates $w=(w_1,w')$ near $y_0$
such that $f(M)=\{w_1=0\}$. Let $z_1$ be a defining function of $M$ and
set
\[
    z_j:=w_j\circ f,\qquad 2\leq j\leq n.
\]
Then $z=(z_1,z')$ is a coordinate system near $p$, and
\[
    M=\{z_1=0\},\qquad
    f(z_1,z')=\bigl(g(z_1,z'),z'\bigr),\qquad
    g(0,z')=0.
\]
Moreover, $\det df=\frac{\partial g}{\partial z_1}$. Since the ramification divisor is supported on the connected smooth
hypersurface $M$, after shrinking again, there is a nowhere vanishing holomorphic function $u$ near $M$, such that $\frac{\partial g}{\partial z_1}
=z_1^{m-1}u(z_1,z')$ . Using $g(0,z')=0$, we obtain
\[g(z_1,z')=\int_0^{z_1}t^{m-1}u(t,z')\,dt =
    z_1^mv(z_1,z'),
\]where 
$v(z_1,z')
:=
\int_0^1s^{m-1}u(sz_1,z')ds$. In particular, $v(0,z')=\frac{u(0,z')}{m}\neq0$. After shrinking the polydisc, $v$ is nowhere vanishing and admits a
holomorphic $m$-th root. Replacing $z_1$ by $z_1v^{1/m}$, we obtain
\begin{equation*}
    M=\{z_1=0\},
    \qquad
    f(z_1,z')=(z_1^m,z').
\end{equation*}
Fix a compact subset $K\Subset M\cap D$.
For each $p=(0,p')\in K$, consider the centered
anisotropic rescaling
\[
S_{k,p}(\zeta_1,\zeta')
=
\left(k^{-\frac{1}{2m}}\zeta_1,\,
p'+k^{-\frac12}\zeta'\right).
\]After a suitable normalization of the
holomorphic coordinates, depending smoothly on $p$,
the rescaled model weight takes the form
\[
\phi_{\mathrm{mod},p}(\zeta)
=
c(p)|\zeta_1|^{2m}
+
Q_p(\zeta',\bar\zeta'),
\]
where $c(p)>0$ and $Q_p$ is a positive-definite
Hermitian quadratic form. By compactness, there exists
$c_K>0$ such that
\[
c(p)\ge c_K,\qquad
Q_p(\zeta',\bar\zeta')\ge c_K|\zeta'|^2,
\qquad p\in K.
\]
By the construction in the proof of
Proposition~\ref{model_gap_finite_typ},
we may choose the same bounded auxiliary function
$F$, independently of $p\in K$, so that
\[
i\partial\bar\partial
(\phi_{\mathrm{mod},p}+F)
\ge \varepsilon_K\omega_{0}
\]
for some $\varepsilon_K>0$.
Consequently, the model spectral gap holds uniformly
for $p\in K$. Moreover, after shrinking $D$, the localized mild
spectral gap holds on a common neighborhood of $K$,
since $D\Subset\operatorname{Amp}(f^*L)$.

The coefficients of the rescaled operators and all
their derivatives are uniformly controlled on compact
parameter sets. In particular, the elliptic estimates,
resolvent off-diagonal estimates, and the asymptotic
estimates used in the proofs of
Theorems~\ref{Thm_Loc_Off} and
Theorem \ref{Thm_Full expansion}
can be taken with constants independent of $p\in K$.
Indeed, these estimates involve only finitely many
seminorms of the local metric coefficients, together
with the spectral gap constants established above. It follows that, for every $N\in\mathbb N_0$,
there exists $C_{N,K}>0$ such that
\[
\sup_{p\in K}
\left|
k^{-s}B_k(p,p)
-\sum_{j=0}^N k^{-\frac{j}{2m}}b_j(p)
\right|
\le C_{N,K}k^{-\frac{N+1}{2m}},
\qquad
s=n-1+\frac1m.
\]Finally, since $p\mapsto B_k(p,p)$ is continuous,
the coefficients $b_j$ are continuous. 
\end{proof}

\section{Non-quasihomogeneous models}\label{sec_non}
In this section, we drop the quasihomogeneity assumption. Theorem~\ref{Thm_Loc_Off} still
localizes the Bergman kernel, and we show that the on-diagonal value at $p$ has leading term determined by the Bergman kernel of the model weight on $\Cs^n$.
The asymptotics may contain logarithmic factors. We compute them explicitly for
the projective example in the introduction, and we explain how Kamimoto's work
\cite{kamimoto2004newton} covers more general torus-invariant models.

\subsection{Reduction to the model}
Let $\Lambda$ be a weight vector and let $z=(z_1,\dots,z_n)$ be model coordinates at $p$ as in
Definition~\ref{def:model-gap}. We denote by $\varphi_\Lambda(z):=\sum_{|(\alpha,\beta)|_\Lambda\leq1}a_{\alpha,\beta}z^\alpha\bar z^\beta$ the $\Lambda$-truncation of the local weight, which need not be quasihomogeneous. By
construction,
\begin{equation}\label{eq:scaling_identity}
k\,\varphi_\Lambda\bigl(\Sk(z)\bigr)=\phimodk(z),\qquad z\in\Cs^n .
\end{equation}
Let $\omega_0$ be the constant Hermitian form obtained by freezing $\omega$ at $p$, and let
$B_{k\varphi_\Lambda,\omega_0}(z,w)$ be the Bergman kernel of the weighted Fock space
$\{f\in\mathcal O(\Cs^n);\int_{\Cs^n}|f|^2e^{-2k\varphi_\Lambda}\frac{\omega_0^n}{n!}<\infty\}$, in the
convention \eqref{Bk2}. Let $d_1>0$ be the smallest positive order in Definition~\ref{order}.

\begin{theorem}[Reduction to the model]\label{Thm_non}
Assume that $\Box_k$ has a localized mild spectral gap on an open set $D\ni p$, and that $p$
satisfies the model gap condition with weight vector $\Lambda$ in the model coordinates $z$.
Then, for every $\delta>0$,
\begin{equation}\label{eq:reduction}
B_k(p,p)=\bigl(1+O(k^{-d_1+\delta})\bigr)\,B_{k\varphi_\Lambda,\omega_0}(0,0)+O(k^{-\infty}).
\end{equation}
In particular, if $B_{k\varphi_\Lambda,\omega_0}(0,0)\geq ck^{-M}$ for some $c,M>0$, then
\[
B_k(p,p)=B_{k\varphi_\Lambda,\omega_0}(0,0)\,\bigl(1+O(k^{-d_1+\delta})\bigr).
\]
for all sufficiently large $k$. 
\end{theorem}

\begin{proof}
By Theorem~\ref{Thm_Loc_Off},  for any $N\in\N$,
\begin{equation}\label{eq:red1}
\bigl|k^{-2|\Lambda|}B_k(p,p)-\tBk(0,0)\bigr|\leq C_Nk^{-N}.
\end{equation}
Recall that $\tBk$ is the Bergman kernel for the measure $e^{-2\tphik}dV_{\tomegak}$. Let
\[
\psi_k:=\tphik-\frac12\log\frac{dV_{\tomegak}}{dV_{\omega_0}}.
\]
We may regard $\tBk$ as the Bergman kernel for the measure $e^{-2\psi_k}dV_{\omega_0}$. All kernels
here are taken in the convention \eqref{Bk2}, with respect to the same measure
$dV_{\omega_0}$. By \eqref{eq:scaling_identity} and the change of variables $z\mapsto\Sk(z)$,
\begin{equation}\label{eq:red2}
k^{-2|\Lambda|}B_{k\varphi_\Lambda,\omega_0}(0,0)=B_{\phimodk,\omega_0}(0,0).
\end{equation}
We observe that for every $\delta>0$, after choosing $\epsilon_0>0$ in \eqref{epsilon_0} small enough,
\begin{equation}\label{eq:red3}
\sup_{z\in\Cs^n}\bigl|\psi_k(z)-\phimodk(z)\bigr|\leq Ck^{-d_1+\delta}.
\end{equation}
Now, we compare the two Fock spaces. The extremal characterization of the Bergman
kernel on the diagonal gives
\[
B_{\psi_k,\omega_0}(0,0)=\sup\Bigl\{\frac{|f(0)|^2}{\int_{\Cs^n}|f|^2e^{-2\psi_k}dV_{\omega_0}}\Bigr\},
\qquad
B_{\phimodk,\omega_0}(0,0)=\sup\Bigl\{\frac{|f(0)|^2}{\int_{\Cs^n}|f|^2e^{-2\phimodk}dV_{\omega_0}}\Bigr\},
\]
where the suprema are over nonzero entire functions with finite norm. Let  $\varepsilon_k$ denote the left-hand side of
\eqref{eq:red3}. We have 
\[
e^{-2\varepsilon_k}B_{\phimodk,\omega_0}(0,0)\leq\tBk(0,0)\leq e^{2\varepsilon_k}B_{\phimodk,\omega_0}(0,0).
\]
Combining this with \eqref{eq:red1}, \eqref{eq:red2} and \eqref{eq:red3}, and multiplying by
$k^{2|\Lambda|}$, we obtain \eqref{eq:reduction}.
\end{proof}
\subsection{A projective example}\label{subsec:example}
We return to the example in the introduction. Let $X=\Ps^2$, $L=\mathcal O_{\Ps^2}(1)$ with the
metric $h$ defined there, $\omega=\omega_{\mathrm{FS}}=\ii\partial\db\log(|Z_0|^2+|Z_1|^2+|Z_2|^2)$,
and $p=[1:0:0]$. In the chart $U_0$,
\[
\varphi(z)=\frac16\log\bigl(1+3|z_1z_2|^2+3|z_1|^6+3|z_2|^6\bigr)
=\underbrace{\frac12\bigl(|z_1z_2|^2+|z_1|^6+|z_2|^6\bigr)}_{=\varphi_\Lambda}+O(|z|^8),
\]
with $\Lambda=(\frac16,\frac16)$, and $\omega_0=\ii\sum_{j=1}^2dz_j\wedge d\bar z_j$, so that
$\omega_0^2/2!=4\,dV$. Here $d_1=\frac16$.

\begin{proposition}\label{prop:example}
With the above notation,
\begin{equation}\label{eq:example_exact}
B_k(p,p)=\frac{3}{4\pi^2}\,\frac{k}{\log k+\gamma}\,\bigl(1+O(k^{-\frac16+\delta})\bigr)
\end{equation}
for every $\delta>0$, where $\gamma$ is the Euler--Mascheroni constant. In particular, for every
$L\in\N_0$,
\[
B_k(p,p)=\frac{3}{4\pi^2}\,\frac{k}{\log k}\sum_{\ell=0}^{L}\Bigl(-\frac{\gamma}{\log k}\Bigr)^{\ell}
+O\Bigl(\frac{k}{(\log k)^{L+2}}\Bigr).
\]
\end{proposition}

\begin{proof}
Since $L$ is ample, $\operatorname{Amp}(L)=X$. The point $p$ is of good finite type with
$\mathcal M=(3,3)$, $P_j=\frac12|z_j|^6$ and $Q=\frac12|z_1z_2|^2$, so the model gap condition
holds by Proposition~\ref{model_gap_finite_typ}. By Theorem~\ref{Thm_non}, it suffices to show
\begin{equation}\label{eq:example_model}
B_{k\varphi_\Lambda,\omega_0}(0,0)=\frac{3}{4\pi^2}\,\frac{k}{\log k+\gamma}\,\bigl(1+O(k^{-1})\bigr).
\end{equation}
Since $\varphi_\Lambda$ is invariant under $(z_1,z_2)\mapsto(e^{\ii\theta_1}z_1,e^{\ii\theta_2}z_2)$, the
monomials are mutually orthogonal, and only the constant function is nonzero at the origin.
Hence $B_{k\varphi_\Lambda,\omega_0}(0,0)=I_k^{-1}$, where, with $x=|z_1|^2$, $y=|z_2|^2$ and 
$u=k^{1/3}x$, $v=k^{1/3}y$,
\[
I_k:=\int_{\Cs^2}e^{-2k\varphi_\Lambda}\,\frac{\omega_0^2}{2!}
=4\pi^2\int_0^\infty\!\!\int_0^\infty e^{-k(xy+x^3+y^3)}\,dx\,dy
=4\pi^2k^{-2/3}J(k^{1/3}),
\]
\[
J(\lambda):=\int_0^\infty\!\!\int_0^\infty e^{-\lambda uv-u^3-v^3}\,du\,dv .
\]
We compute the asymptotics of $J$ by the Mellin transform. For $0<\operatorname{Re}s<1$, Fubini's
theorem and $\int_0^\infty\lambda^{s-1}e^{-\lambda uv}d\lambda=\Gamma(s)(uv)^{-s}$ give
\[
\mathcal MJ(s):=\int_0^\infty\lambda^{s-1}J(\lambda)\,d\lambda
=\Gamma(s)\Bigl(\int_0^\infty u^{-s}e^{-u^3}du\Bigr)^2
=\frac19\,\Gamma(s)\,\Gamma\Bigl(\frac{1-s}{3}\Bigr)^2 .
\]
This function extends meromorphically to $\Cs$. By Stirling's formula, it decays exponentially
on vertical lines, and in $\operatorname{Re}s>0$ its only poles are double poles at
$s=1,4,7,\dots$.
By Mellin inversion formula, we have \[
J(\lambda)=\frac{1}{2\pi\ii}\int _{c-\ii\infty}^{c+\ii\infty}\lambda^{-z}\frac{1}{9}\Gamma(z)\Gamma(\frac{1-z}{3})^2 dz,\quad 0<c<1.
\]
Shifting the contour from
$\operatorname{Re}s=\frac12$ to $\operatorname{Re}s=5$ crosses the poles at $s=1$ and $s=4$. This gives
\[
J(\lambda)=-\operatorname{Res}_{s=1}\bigl(\lambda^{-s}\mathcal MJ(s)\bigr)+O(\lambda^{-4}\log\lambda),
\]
where the contribution of the double pole at $s=4$ is $O(\lambda^{-4}\log\lambda)$ and the remaining
integral is $O(\lambda^{-5})$.
Writing $s=1+t$ and using $\Gamma(1+t)=1-\gamma t+O(t^2)$ and
$\Gamma(-\frac t3)=-\frac3t-\gamma+O(t)$, we get
\[
\mathcal MJ(1+t)=\frac1{t^2}-\frac{\gamma}{3}\,\frac1t+O(1),
\qquad
\lambda^{-1-t}=\lambda^{-1}\bigl(1-t\log\lambda+O(t^2)\bigr).
\]
Therefore,
\[
J(\lambda)=\frac{\log\lambda+\gamma/3}{\lambda}+O(\lambda^{-4}\log\lambda).
\]
With $\lambda=k^{1/3}$, this gives $I_k=\frac{4\pi^2}{3}\,\frac{\log k+\gamma}{k}\,\bigl(1+O(k^{-1})\bigr)$,
which is \eqref{eq:example_model}. The second statement follows by expanding
$(\log k+\gamma)^{-1}$ in powers of $(\log k)^{-1}$.
\end{proof}
\subsection{General torus-invariant models}
Theorem~\ref{Thm_non} reduces the on-diagonal asymptotics at $p$ to those of
$B_{k\varphi_\Lambda,\omega_0}(0,0)$.
For torus-invariant models, this problem was solved by Kamimoto
\cite{kamimoto2004newton}, who studied the Bergman kernel of the model domains
$\{(z_0,z)\in\Cs\times\Cs^n;\operatorname{Im}z_0>F(z)\}$ via Haslinger's integral
representation \cite{haslinger1998bergman}. We only state the consequence we need. The Newton
distance $d_F$ and the multiplicity $m_F$ of the Newton polyhedron of $F$ are defined in
\S2.2 of that paper.

Assume that $\Lambda=(\frac1{2m_1},\dots,\frac1{2m_n})$ and that
\begin{equation}\label{good_model}
\varphi_\Lambda(z)=\sum_{j=1}^na_j|z_j|^{2m_j}+\sum_{\substack{\alpha\in\N_0^n\ \text{mixed}\\
2|\alpha|_\Lambda\leq1}}a_\alpha|z^\alpha|^2,\qquad a_j>0,\ a_\alpha\geq0 .
\end{equation}
Then, $p$ is of good finite type, and $F=\varphi_\Lambda$ satisfies the finite-type,
torus-invariance and growth assumptions of \cite[Theorem~3.6]{kamimoto2004newton}. Since
$1\in L^2$ for the weight $e^{-2k\varphi_\Lambda}$, the lower bound in Theorem~\ref{Thm_non}
holds, and that theorem yields the following.

\begin{corollary}\label{cor:kamimoto}
Assume that $\Box_k$ has a localized mild spectral gap near $p$ and that the model at $p$ has
the form \eqref{good_model}. Then there are real numbers $b_{0,\ell}$, $\ell\in\N_0$, with
$b_{0,0}>0$ depending only on the Newton principal part of $\varphi_\Lambda$ and on $\omega(p)$,
such that for every $L\in\N_0$,
\[
B_k(p,p)=\frac{k^{2/d_{\varphi_\Lambda}}}{(\log k)^{m_{\varphi_\Lambda}-1}}
\Bigl(\sum_{\ell=0}^{L}b_{0,\ell}(\log k)^{-\ell}+O\bigl((\log k)^{-L-1}\bigr)\Bigr).
\]
\end{corollary}
In the example of \S\ref{subsec:example}, the Newton polyhedron of $\varphi_\Lambda$ meets the diagonal
at the vertex $(2,2)$, so $d_{\varphi_\Lambda}=2$ and $m_{\varphi_\Lambda}=2$, in accordance with
\eqref{eq:example_exact}.